\documentclass{article}
\usepackage[utf8]{inputenc}
\usepackage[english]{babel}

\usepackage{amsmath,amsthm,amsfonts,amssymb,amscd,lastpage,fancyhdr,mathrsfs,enumerate,polynom,graphicx,tikz,hyperref,csquotes,verbatim,accents,comment,float,pinlabel}
 
\usepackage{biblatex} 
\usepackage{geometry}
\newcommand{\N}{\mathbb{N}}

\newcommand{\flip}{\mathcal{F}(S)}
\newcommand{\flipp}{\mathcal{F}(S')}
\newcommand{\emod}{\text{Mod}^\pm(S)}

\theoremstyle{plain}
\newtheorem{theorem}{Theorem}[section]
\newtheorem{lemma}[theorem]{Lemma}
\newtheorem{prop}[theorem]{Proposition}
\newtheorem{corollary}[theorem]{Corollary}

\theoremstyle{remark}

\newtheorem*{acknowledgements}{Acknowledgements}

\begin{document}

\title{Finite Rigid Sets in Flip Graphs}
\author{Emily Shinkle} 
\date{}

\maketitle 

\thispagestyle{fancy}
\fancyhf{}
\lfoot{\small 2020 \textit{Mathematics Subject Classification}. Primary: 57K20, 20F65. Secondary: 57M60, 05C60, 05C25.}
\renewcommand{\headrulewidth}{0pt}
\renewcommand{\footrulewidth}{1pt}

\begin{abstract} 
We show that for most pairs of surfaces, there exists a finite subgraph of the flip graph of the first surface so that any injective homomorphism of this finite subgraph into the flip graph of the second surface can be extended uniquely to an injective homomorphism between the two flip graphs. Combined with a result of Aramayona-Koberda-Parlier, this implies that any such injective homomorphism of this finite set is induced by an embedding of the surfaces.
\end{abstract}

\section{Introduction}\label{section: introduction}

Let $S$ and $S'$ be compact, connected, orientable surfaces, possibly with boundary, with a nonzero number of marked points and at least one marked point on each boundary component. A triangulation of $S$ is a maximal collection of isotopy classes of essential arcs on $S$ with pairwise disjoint representatives. The \emph{flip graph} $\flip$ of $S$, also known as the \emph{ideal triangulation graph}, is the graph whose vertices represent triangulations of $S$ and whose edges connect vertices representing triangulations which differ by one arc. 

Embeddings $S \rightarrow S'$ induce injective graph homomorphisms $\flip \rightarrow \flipp$. In \cite{aramayona2015injective}, Aramayona-Koberda-Parlier prove that except for a finite number of homeomorphism types of surfaces $S$, all injective homomorphisms $\flip \rightarrow \flipp$ arise in this way. This result is known as the \emph{rigidity} of $\flip$ and similar rigidity results have been proved for many other simplicial complexes associated to surfaces. Recently, there has been interest in proving \emph{finite rigidity} results for these complexes as well. In this paper, we prove a finite rigidity result for $\flip$. Here, we say that a finite subgraph $\mathcal{X} \subseteq \flip$ is a \emph{finite rigid set for the pair $(S,S')$} if any injective homomorphism $\lambda:\mathcal{X} \rightarrow \flipp$ can be extended uniquely to an injective homomorphism $\phi : \flip \rightarrow \flipp$, i.e. $\phi|_{\mathcal{X}} = \lambda$. 

We write $S_{g,n,(p_1,\ldots,p_b)}$ for a surface of genus $g$ with $n$ marked points in its interior and $b$ boundary components where the number of marked points on the $i^{th}$ boundary component is $p_i$. If $b  =  0$, we write $S_{g,n}$. 

\begin{theorem} \label{theorem: main result}
Let $S\not \cong $ $S_{0,0,(1,1)}, S_{1,1}$. Let $S'$ be any surface. Then there exists a finite rigid set for the pair $(S,S')$. 
\end{theorem}

The case $S\cong S_{0,0,(1,1)}$ is more complicated, but we can decide precisely when finite rigid sets exist.

\begin{theorem} \label{theorem: main result add-on} If $S\cong S_{0,0,(1,1)}$, then there is a finite rigid set for the pair $(S,S')$ if and only if $\flipp$ is finite or $S'\cong S_{0,0,(1,1)}$. 
\end{theorem}

The final case $S\cong S_{1,1}$ presents some additional difficulties, which we discuss in Section \ref{section: remaining case}. At present, we do not know in general whether finite rigid sets exist in this case.

The next corollary follows from combining Theorem \ref{theorem: main result} with the result of Aramayona-Koberda-Parlier in \cite{aramayona2015injective}. They call $S$ \emph{exceptional} if it is a (not necessarily proper) essential subsurface of $S_{0,n}$ with $n\leq 4$ or $S_{1,n}$ with $n\leq 2$, and \emph{non-exceptional} otherwise. We will do the same. Define $d(S) = 6g+3b+3n+p_1+p_2+\ldots +p_b-6$. 

\begin{corollary}\label{corollary: applying a-k-p theorem}
Suppose $S$ is non-exceptional and $S'$ is any surface. Then there exists a finite subgraph $\mathcal{X}\subseteq \flip$ such that any injective homomorphism $\lambda:\mathcal{X} \rightarrow \flipp$ is induced by an embedding $h:S \rightarrow S'$, unique up to isotopy. If $d(S) = d(S')$, then $h$ is a homeomorphism. 
\end{corollary}

Finally, we note that our construction proves the following.

\begin{corollary}\label{corollary: exhaustion}
If there exists a finite rigid set for the pair $(S,S')$, then there exists an exhaustion of $\flip$ by finite rigid sets for the pair $(S,S')$, i.e. a sequence $(\mathcal{X}_i)_{i\in \N}$ of finite rigid sets for $(S,S')$ such that $\mathcal{X}_0\subseteq \mathcal{X}_1 \subseteq ... \subseteq \flip$ and $\bigcup_{i\in \N} \mathcal{X}_i  =  \flip$. 
\end{corollary}

The flip graph is primarily of interest because of its connections to mapping class groups. The \emph{extended mapping class group} of $S$, $\emod$, is the group of homeomorphisms of $S$ up to isotopy. Properties of $\emod$ can be better understood by studying its actions on certain simplicial complexes (see e.g. \cite{harer1985stability}, \cite{hamenstadt2005word}).  Many of these simplicial complexes have been shown to be \emph{rigid}, meaning that all of their automorphisms are induced by homeomorphisms. If such a homeomorphism is always unique, this implies that the automorphism group of the complex is isomorphic to $\emod$. 

Ivanov proved the first such rigidity result for the curve complex of many surfaces in \cite{ivanov1997automorphisms}. His result for the curve complex was extended by Korkmaz \cite{korkmaz1999automorphisms} and Luo \cite{luo1999automorphisms} to cover more surfaces, and strengthened by Shackleton \cite{shackleton2007combinatorial} and Hernández-Hernández \cite{hernandez2018edge} to apply for maps between differing curve complexes. Rigidity results also exist for the arc complex (Irmak-McCarthy \cite{irmak2010injective}, Disarlo \cite{disarlo2015combinatorial}), the pants complex (Margalit \cite{margalit2004automorphisms}, Aramayona \cite{aramayona2010simplicial}), the arc and curve complex (Korkmaz--Papadopoulos \cite{korkmaz2010arc}), the Schumtz graph of nonseparating curves (Schaller \cite{schaller2000mapping}), the complex of nonseparating curves (Irmak \cite{irmak2004complexes}), the Hatcher-Thurston complex (Irmak--Korkmaz \cite{irmak2007automorphisms}), the polygonalisation complex (Bell--Disarlo--Tang \cite{bell2019cubical}), and more. Some of these results have been extended to the case of non-orientable surfaces, such as the curve complex (Atalan--Korkmaz \cite{atalan2014automorphisms}, Irmak \cite{irmak2014simplicial}, Irmak \cite{irmak2012superinjective}), the arc complex (Irmak \cite{irmak2008injective}), and the two-sided curve complex (Irmak--Paris \cite{irmak2017superinjective}). 

Ivanov conjectured that every object naturally associated to a surface with sufficiently rich structure has an automorphism group isomorphic to the extended mapping class group \cite{ivanov2006fifteen}. Some work towards proving this has been done by Brendle--Margalit \cite{brendle2019normal} and McLeay \cite{mcleay2018geometric}, who prove general rigidity results about subcomplexes of the complex of domains, which was first introduced by McCarthy--Papadopoulos in \cite{mccarthy2012simplicial}.

In some cases, such results can be retained for maps defined only on certain finite subcomplexes, called \textit{finite rigid sets}. This type of result was originally suggested by Lars Louder. Such has been shown for the curve complex (Aramayona--Leininger \cite{aramayona2013finite}, Ilbira--Korkmaz \cite{ilbira2018finite}, Irmak \cite{irmak2019exhausting} and \cite{irmak2019exhausting2}), the arc complex (the author \cite{shinkle2020finite}), and the pants graph (Maungchang \cite{maungchang2018finite}, Hernández-Hernández--Leininger--Maungchang \cite{maungchang2019finite}, Maungchang \cite{maungchang2017exhausting}). We prove such a result for the flip graph in this paper.

\begin{acknowledgements}
The author would like to thank Christopher Leininger for suggesting this project and for providing guidance and support throughout it. She would also like to thank the referee for their careful reading and helpful suggestions.
\end{acknowledgements}

\section{Basics}\label{section: basics} 

For the remainder of the paper, $S$ and $S'$ will refer to compact, connected, orientable surfaces, possibly with boundary, with a nonzero number of marked points and at least one marked point on each boundary component. Let $\mathcal{P}_S$ denote the set of all marked points of $S$, including those in the interior of $S$ and those in $\partial S$. We require that a homeomorphism $f : S \rightarrow S'$ induces a bijection $\mathcal{P}_S \rightarrow \mathcal{P}_{S'}$. We will consider surfaces up to homeomorphism. 

An \textit{arc} on $S$ is a map $a:[0,1]\rightarrow S$ such that $a(0)$ and $a(1)$ (called the \textit{endpoints} of $a$) are in $\mathcal{P}_S$, and $a|_{(0,1)}$ (called the \textit{interior} of $a$) is injective and disjoint from $\mathcal{P}_S$. We identify arcs with their image on $S$. We consider arcs up to isotopy where each path in the one-parameter family of the isotopy is also an arc. In particular, isotopies of arcs are relative to the marked points and interiors do not pass though marked points. We also assume that arcs are essential, meaning they cannot be isotoped in an arbitrarily small neighborhood of a marked point or into $\partial S$. See Figure \ref{figure: examples of essential and nonessential arcs}. 

\begin{figure}[h]
    \labellist
        \pinlabel {Nonessential arcs} at 56 9
        \pinlabel {Essential arcs} at 170 9
    \endlabellist
    \centering
    \includegraphics{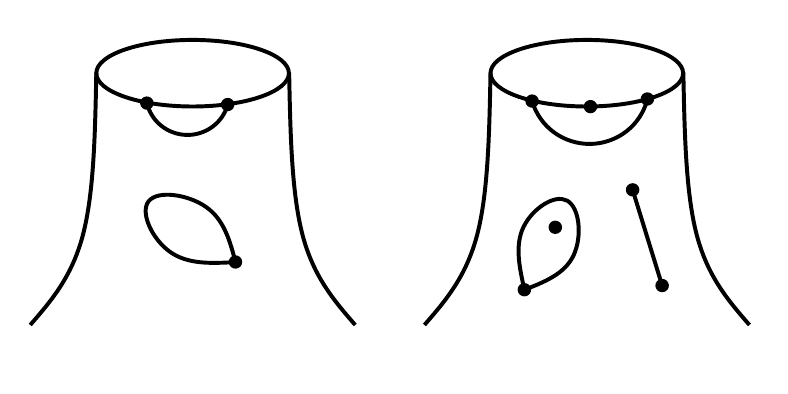}
    \caption{Examples of essential and nonessential arcs}
    \label{figure: examples of essential and nonessential arcs}
\end{figure}

We say that a continuous map $h : S \rightarrow S'$ is an \emph{embedding} if there exists a (possibly empty) collection of arcs with pairwise disjoint interiors $A\subseteq S'$ such that $h$ restricts to a homeomorphism from $\text{Int}(S)$ to a component of $S'\backslash A$, and $h(\mathcal{P}_S)  =  \mathcal{P}_{S'}\cap h(S)$. Note that $S$ is determined, up to homeomorphism, by the component of $S'\backslash A$ containing $h(S)$. If there exists an embedding from $S$ to $S'$, we say that $S$ is an \emph{essential subsurface} of $S'$. 

The \emph{(geometric) intersection number}, $i(a,b)$, of two arcs $a$ and $b$ is the minimum number of intersections of the interiors of representatives of $a$ and $b$. We say $a$ and $b$ are \textit{disjoint} if $i(a,b) = 0$.  

A \textit{triangulation} $T$ of $S$ is a maximal collection of distinct, pairwise-disjoint arcs. If $S \cong S_{0,1}, S_{0,0,(1)}, S_{0,0,(2)}, S_{0,0,(3)}$, then $S$ admits no arcs, so the only triangulation is the empty set. If $S\cong S_{0,2}$, there is a single triangulation, consisting of a single arc. Otherwise, a standard argument using the Euler characteristic shows that every triangulation of $S\cong S_{g,n,(p_1, \ldots, p_b)}$ contains $d(S) = 6g+3b+3n+p_1+p_2+\cdots + p_b -6$ arcs. Further, we can choose disjoint representatives of the arcs in a triangulation $T$ and then $T$ divides $S$ into \textit{triangles}. Formally, these triangles are embeddings $\Delta:S_{0,0,(3)} \rightarrow S$. We identify these triangles with their images on $S$. The boundary of $S_{0,0,(3)}$ is composed of three segments, each starting and ending at one of the marked points. We call the images of these segments under $\Delta$ the \emph{sides} of $\Delta$ and say that these segments \emph{border} $\Delta$. These sides are either arcs on $S$ or segments of $\partial S$. If the three sides of $\Delta$ are distinct, we call $\Delta$ \textit{non-folded}. Otherwise, we call $\Delta$ \textit{folded}. If $\Delta$ is folded, we call the two sides which are identified the \emph{inner side} of $\Delta$ and we call the other side the \emph{outer side}. See Figure \ref{figure: examples of triangles in a triangulation} for examples. 

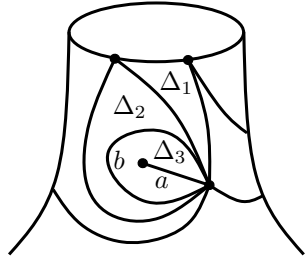
\begin{figure}[h]
    \labellist
        \pinlabel {$\Delta_1$} at 64.5 74
        \pinlabel {$\Delta_2$} at 47 65
        \pinlabel {$\Delta_3$} at 62 48
        \pinlabel {$a$} at 59 36.5
        \pinlabel {$b$} at 43 45
    \endlabellist
    \centering
    \includegraphics{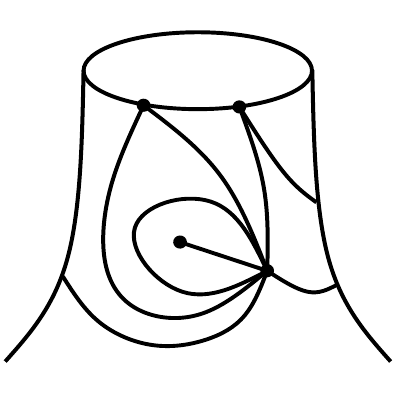}
    \caption{Examples of triangles. The triangles $\Delta_1$ and $\Delta_2$ are non-folded, whereas the triangle $\Delta_3$ is folded. The arc $a$ is the inner side of the folded triangle $\Delta_3$ and $b$ is the outer side.}
    \label{figure: examples of triangles in a triangulation}
\end{figure}

We say that two triangulations $T$ and $T'$ of a surface $S$ \textit{differ by a flip} if $T$ and $T'$ share exactly $d(S)-1$ arcs. Note that this implies that the two arcs on which they differ have intersection number one. See Figure \ref{figure: illustration of a flip}. We say $T'$ \emph{is a flip of $T$ along $a$} if $T$ and $T'$ differ by a flip and if $a$ is the arc in $T$ which is not in $T'$. If $a$ is the inner side of a folded triangle in $T$, then no triangulations differ from $T$ by a flip along $a$, and we say $a$ is \textit{not flippable}. If $a$ is any other arc, there is exactly one triangulation which is a flip of $T$ along $a$. We will sometimes denote this triangulation by $T_a$ and we say that $a$ is \textit{flippable}.

\begin{figure}[h]
    \labellist
        \pinlabel {$a$} at 55 68
        \pinlabel {$a'$} at 33 42
    \endlabellist
    \centering
    \includegraphics{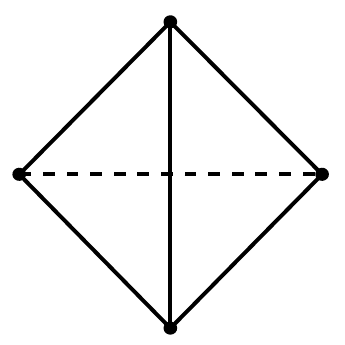}
    \caption{If a triangulation containing the solid arcs is flipped along the arc $a$, the resulting triangulation will contain $a'$ instead of $a$.}
    \label{figure: illustration of a flip}
\end{figure}

The \textit{flip graph} $\flip$ of $S$ is a graph whose vertices represent triangulations of $S$ and whose edges connect vertices representing triangulations which differ by a flip. This graph is sometimes called the \textit{ideal triangulation graph} by other authors. We use the term \emph{homomorphism} to refer to a simplicial map between two graphs, such as flip graphs. 
 
Aramayona-Koberda-Parlier explain in \cite{aramayona2015injective} how an embedding $h:S \rightarrow S'$ induces an injective homomorphism $\phi : \flip \rightarrow \flipp$: choose $Q$ to be a triangulation of $S'\backslash \text{int}(h(S))$ plus a collection of arcs on $\partial h(S)$ whose union is homeomorphic to $\partial h(S)$. Then define $\phi(T) = h(T)\cup Q$ for each vertex $T\in \flip$. Aramayona-Koberda-Parlier prove for all but finitely many homeomorphism classes of surfaces $S$, that all injective homomorphisms $\flip \rightarrow \flipp$ arise in this way. Recall that we say $S$ is exceptional if it is a (not necessarily proper) essential subsurface of $S_{0,n}$ with $n\leq 4$ or $S_{1,n}$ with $n\leq 2$, and non-exceptional otherwise. 

\begin{theorem}[\cite{aramayona2015injective}, Theorems 1.1, 1.4] \label{theorem: aramayona-koberda-parlier} Suppose $S$ is non-exceptional. Then every injective homomorphism $\phi :  \flip \rightarrow \flipp$ is induced by an embedding $h:S \rightarrow S'$. Further, if $d(S) = d(S')$, $h$ is a homeomorphism. 
\end{theorem}

Korkmaz and Papadopoulus also showed this in the case when $S\cong S'$ and $\partial S = \emptyset$ in \cite{korkmaz2012ideal}, using a different method.

\section{Preliminary Results}

We begin by proving a weakened version of our theorem. It uses only three properties of flip graphs: Flip graphs are connected (eg. \cite{mosher1988tiling}). Each vertex in $\flip$ has degree at most $d(S)$, so $\flip$ is locally finite. Finally, there are finitely many $\text{Aut}(\flip)$-orbits of vertices in $\flip$ since, up to homeomorphism, there are only finitely many ways to construct $S$ from a finite collection of triangles. We state the result for any graphs with these properties, but will apply it for $X = \flip$ and $Y = \flipp$.

We can think of any connected graph $Z$ as a metric space on the set of vertices where the distance $d_Z(z,z')$ between vertices $z$ and $z'$ is the minimum number of edges on a path between $z$ and $z'$. Let $B_Z(z;n)$ denote the closed ball of radius $n$ in $Z$ centered at $z$.

\begin{prop}\label{prop: existence of weakly rigid set}
Let $X$ and $Y$ be connected graphs such that $Y$ is locally finite and there are finitely many $\text{Aut}(Y)$-orbits of vertices in $Y$. Let $x\in X$ and $r\geq 0$. Then there exists $R\geq r$ such that for any injective homomorphism
\[\lambda:B_{X}(x;R)\rightarrow Y,\]
there exists an injective homomorphism $\phi : X\rightarrow Y$ such that $\lambda|_{B_{X}(x;r)} = \phi|_{B_{X}(x;r)}$. 
\end{prop}

\begin{proof}
Fix $x\in X$ and $r\geq 0$. Assume for contradiction that the statement does not hold. Then for all $n\geq r$, choose an injective homomorphism $\lambda_n:B_{X}(x;n)\rightarrow Y$ such that for any injective homomorphism $\psi : X \rightarrow Y$, we have $\lambda_n|_{B_{X}(x;r)}\neq \psi|_{B_{X}(x;r)}$. 

Set $y = \lambda_{r}(x)$. Since there are finitely many $\text{Aut}(Y)$-orbits of vertices in $Y$, there is $k\in \N$ such that for any $z, z' \in Y$, there exists $f\in \text{Aut}(Y)$ such that $d_Y(f(z),z')\leq k$. Fix $n>r$. In particular, there is $g\in \text{Aut}(Y)$ such that $d_Y(g(\lambda_n(x)),y)\leq k$. Thus, after composing with an element of $\text{Aut}(Y)$ if necessary, we may assume that $\lambda_n$ maps $x$ within $k$ of $y$ for all $n\geq r$. Note that closed balls in $X$ are connected, and thus each map $\lambda_n$ is 1-Lipschitz, as is any graph homomorphism with connected domain. Thus we conclude that $\lambda_n(B_{X}(x;l))\subseteq B_{Y}(y;l+k)$ for all $l\geq r$ and $n\geq l$. 

Since $Y$ is locally finite and closed balls in $Y$ are connected, any closed ball in $Y$ contains finitely many vertices. Thus by the pigeonhole principle, there exist infinitely many $\lambda_n$ which agree on $B_{X}(x;r)$. Define $\phi|_{B_{X}(x;r)}$ to agree with these maps. Of these maps, infinitely many agree on $B_{X}(x;r+1)$, so define $\phi|_{B_{X}(x;r+1)}$ to agree with them. Repeat inductively to define $\phi : X\rightarrow Y$ which is an injective homomorphism and agrees with $\lambda_n$ on $B_{X}(x;r)$ for infinitely many $n \geq r$. This is a contradiction. 
\end{proof}

To prove Theorem \ref{theorem: main result}, we apply Proposition \ref{prop: existence of weakly rigid set} to $X = \flip$, $Y = \flipp$, and $x=U$ for a vertex $U$ in $\flip$. We then attempt to show that $\lambda|_{B_{\flip}(U;r)} = \phi|_{B_{\flip}(U;r)}$ implies $\lambda|_{B_{\flip}(U;R)} = \phi|_{B_{\flip}(U;R)}$. We show in Section \ref{section: proof of main theorem if d(S) = d(S')} that this is indeed true if $d(S) = d(S')$ and $S$ is non-exceptional. Thus, in these cases, $B_{\flip}(U;R)$ is the required finite rigid set for $(S,S')$. In Section \ref{section: proof of main theorem}, we make minor alterations to the strategy to apply more generally. In both cases, we show that the two maps which were originally known only to agree on $B_{\flip}(U;r)$ actually agree on a larger set of vertices. We do this by using Corollary \ref{corollary: rigid 4-cycles and 5-cycles}, which follows immediately from Lemma \ref{lemma: classification of paths of length 2} below. Recall that $T_a$ refers to a triangulation obtained by flipping another triangulation $T$ along an arc $a$ in $T$.

\begin{lemma}\label{lemma: classification of paths of length 2}
Let $\gamma$ be a path of length 2 in $\flip$. Let $T$ be the middle vertex in $\gamma$, and let $T_a$ and $T_b$ be the other vertices in $\gamma$. Then 
\begin{enumerate}[(i)]
    \item $T$, $T_a$, and $T_b$ are contained in a 4-cycle in $\flip$ if and only if $a$ and $b$ do not border a common triangle in $T$. If there is such a 4-cycle, it is unique. 
    \item $T$, $T_a$, and $T_b$ are contained in a 5-cycle in $\flip$ if and only if $a$ and $b$ border exactly one common triangle in $T$. If there is such a 5-cycle, it is unique. 
    \item $T$, $T_a$, and $T_b$ are contained in no cycles of length less than 6 in $\flip$ if and only if $a$ and $b$ border two common triangles in $T$. 
\end{enumerate}
\end{lemma}

Note that each arc in a triangulation borders at most two triangles, so these cases are exhaustive. Also observe that up to homeomorphism, there are two ways in which $a$ and $b$ can border two common triangles, as in case (iii) and shown in Figure \ref{figure: picture proof two ways for two arcs to border two common triangles}. 

\begin{proof}[Proof of Lemma \ref{lemma: classification of paths of length 2}.] 

Let $a'$ be the arc in $T_a$ which is not in $T$ and $b'$ be the arc in $T_b$ which is not in $T$. See Figure \ref{figure: possible arrangements of two arcs in a triangulation} for the relative positions of $a'$ and $b'$ depending on the arrangements of $a$ and $b$ in $T$. Notice that $i(a',b') = 0$ if and only if $a$ and $b$ do not border a common triangle in $T$, $i(a',b') = 1$ if and only if $a$ and $b$ border exactly one common triangle in $T$, and $i(a',b') = 2$ if and only if $a$ and $b$ border two common triangles in $T$. 

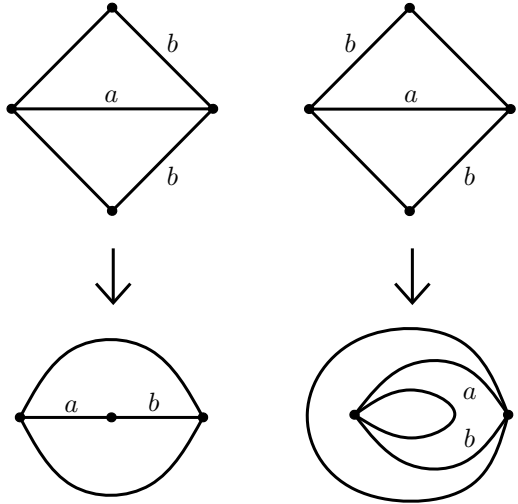
\begin{figure}[h]
    \labellist
        \pinlabel {$a$} at 43 157 
        \pinlabel {$a$} at 156 157 
        \pinlabel {$b$} at 66 177 
        \pinlabel {$b$} at 133 177 
        \pinlabel {$b$} at 66 127 
        \pinlabel {$b$} at 178 127 
        \pinlabel {$a$} at 28 40 
        \pinlabel {$b$} at 59 41.5 
        \pinlabel {$a$} at 178 45 
        \pinlabel {$b$} at 178 28 
    \endlabellist
    \centering
    \includegraphics{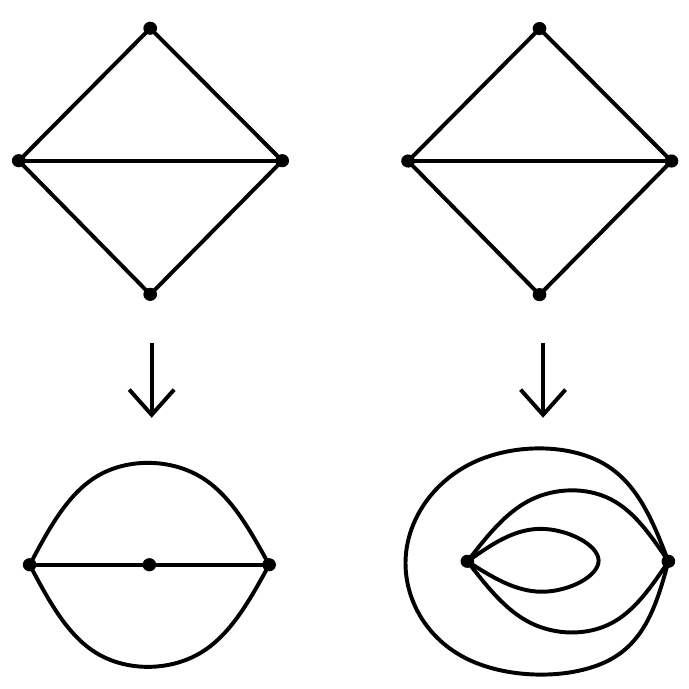}
    \caption{Illustration of the two ways in which two arcs can border two common triangles in a triangulation. Additional vertices and non-labelled sides may be identified in the bottom configurations.}
    \label{figure: picture proof two ways for two arcs to border two common triangles}
\end{figure}

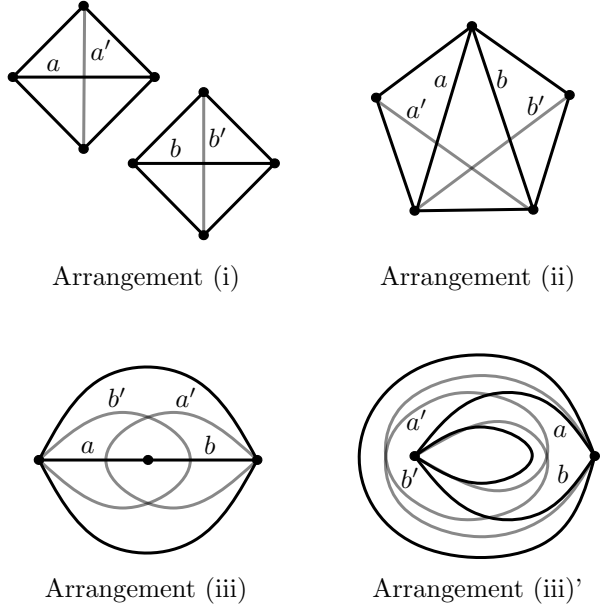
\begin{figure}[h]
    \labellist
        \pinlabel {Arrangement (i)} at 63 138
        \pinlabel {Arrangement (ii)} at 188 138
        \pinlabel {Arrangement (iii)} at 63 20
        \pinlabel {Arrangement (iii)'} at 188 20
        \pinlabel {$a$} at 28 218 
        \pinlabel {$a'$} at 46 224 
        \pinlabel {$b$} at 74 187 
        \pinlabel {$b'$} at 90 191 
        \pinlabel {$a$} at 174 212 
        \pinlabel {$b$} at 197 213 
        \pinlabel {$a'$} at 165 202 
        \pinlabel {$b'$} at 210 203 
        \pinlabel {$a$} at 41 74 
        \pinlabel {$b$} at 87 75 
        \pinlabel {$a'$} at 78 93 
        \pinlabel {$b'$} at 52 93 
        \pinlabel {$a$} at 219 80 
        \pinlabel {$b$} at 220 64 
        \pinlabel {$a'$} at 165 84.5 
        \pinlabel {$b'$} at 162.5 62.5 
    \endlabellist
    \centering
    \includegraphics{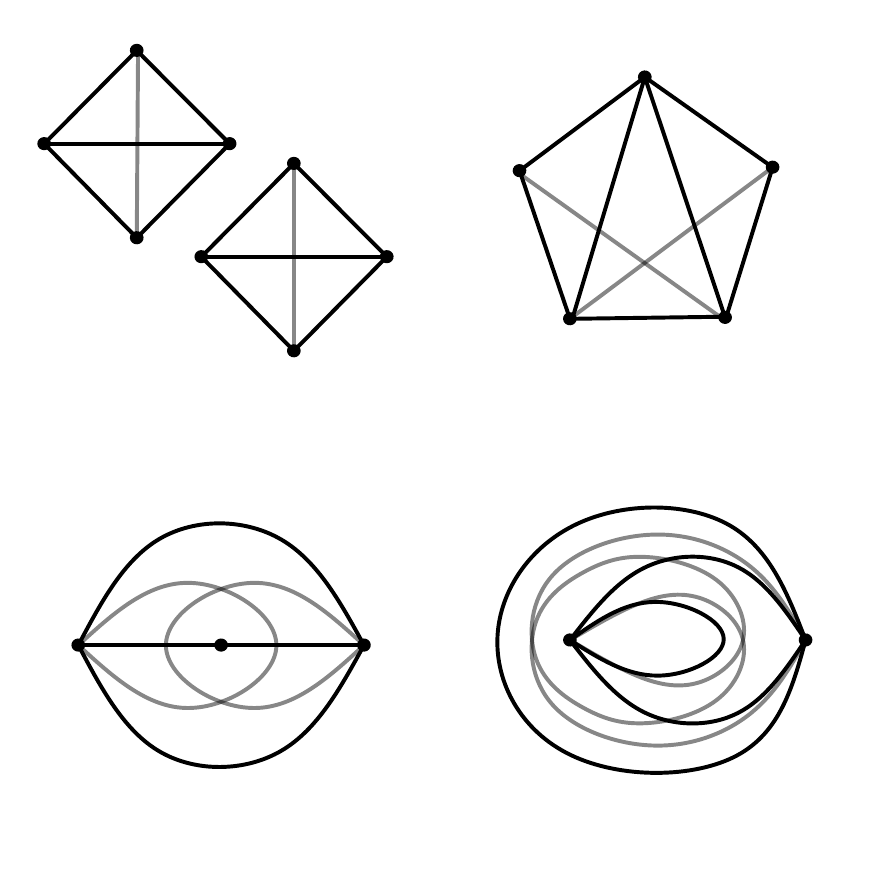}
    \caption{All possible arrangements of two arcs $a$ and $b$ in a triangulation $T$, up to homeomorphism. (There may be additional identification of outer sides and vertices.) The arc $a'$ is the element of $T_a\backslash T$ and the arc $b'$ is the element of $T_b\backslash T$.}
    \label{figure: possible arrangements of two arcs in a triangulation}
\end{figure}

First, we show that there are no 3-cycles in $\flip$, so (iii) follows from (i) and (ii). Let $D = T\backslash \{a,b\}$. Observe that $T_a = D\cup \{a',b\}$ and $T_b = D\cup \{a,b'\}$. In order for $T_a$ and $T_b$ to be connected by an edge in $\flip$, they must agree on all but one arc. However, we can see from Figure \ref{figure: possible arrangements of two arcs in a triangulation} that regardless of the positions of $a$ and $b$ in $T$, we have that $a$, $b$, $a'$, and $b'$ are all distinct.

We now prove (i). If $a$ and $b$ do not border a common triangle, then $T_{ab} = T_{ba}$ (where e.g. $T_{ab} = (T_a)_b$) and this triangulation differs from both $T_a$ and $T_b$ by a flip and is distinct from $T$, so these four vertices form a 4-cycle in $\flip$. Conversely, assume there is a vertex $V$ adjacent to both $T_a$ and $T_b$ which is not $T$. Then $V$ can be obtained by flipping $T_a$ along an arc other than $a'$ or by flipping $T_b$ along an arc other than $b'$. This means that
\[V=(T\backslash \{a,b\}) \cup \{a',b'\}=T_{ab}=T_{ba}\]
and  $i(a',b') = 0$. Thus $a$ and $b$ are in distinct triangles in $T$ (see Figure \ref{figure: possible arrangements of two arcs in a triangulation}(i)) and $V$ is unique. 

Now suppose $a$ and $b$ border exactly one common triangle in $T$. Then $T$, $T_a$, and $T_b$ are contained in the 5-cycle shown in Figure \ref{figure: illustration (a) for the proof of lemma classification of paths of length 2}. Conversely, suppose there is a 5-cycle containing $T$, $T_a$, and $T_b$. Say $V$ (adjacent to $T_a$) and $W$ (adjacent to $T_b$) are the other vertices. See Figure \ref{figure: illustration (b) for the proof of lemma classification of paths of length 2}. Then there are two paths from $T_a$ to $W$. Since one path has length two, we know that $W$ can be obtained from $T_a$ by two flips, hence $W$ and $T_a$ differ by two arcs. The other path has length three and obtains $W$ from $T_a$ by first flipping $T_a$ along $a'$, then flipping the resulting triangulation, $T$, along $b$, and finally flipping the resulting triangulation, $T_b$, along some other arc $c$ in $T_b$ which is not equal to $b'$. Say that $c'$ is the arc which replaces $c$ after this final flip. Thus one of $a'$, $b$, and $c$ must equal one of $a$, $b'$, and $c'$ since $W$ and $T_a$ differ by only two arcs. As $a$, $b$, $a'$, and $b'$ are all distinct, we must have one of $a' = c'$, $b = c'$, $a = c$, or $b' = c$. We cannot have $b = c'$ or $b' = c$ as either of these imply that $T = W$. If both $a' = c'$ and $a = c$, then $W$ differs from $T_a$ by a single flip, but this would result in a 3-cycle. Hence exactly one of $a' = c'$ and $a = c$ holds. 

\begin{figure}[h]
    \labellist
        \pinlabel {$b$} at 21 164 
        \pinlabel {$a$} at 56 166 
        \pinlabel {$b$} at 125 243.5 
        \pinlabel {$a'$} at 155 225 
        \pinlabel {$a'$} at 257 149 
        \pinlabel {$a$} at 95 45 
        \pinlabel {$b'$} at 64.5 25 
        \pinlabel {$b'$} at 193.5 25 
        \pinlabel {$T$} at -2 148
        \pinlabel {$T_b$} at 34 22
        \pinlabel {$T_a$} at 97 227
        \pinlabel {$T_{ab}$} at 288 148
        \pinlabel {$T_{ba}$} at 247 22
    \endlabellist
    \centering
    \includegraphics{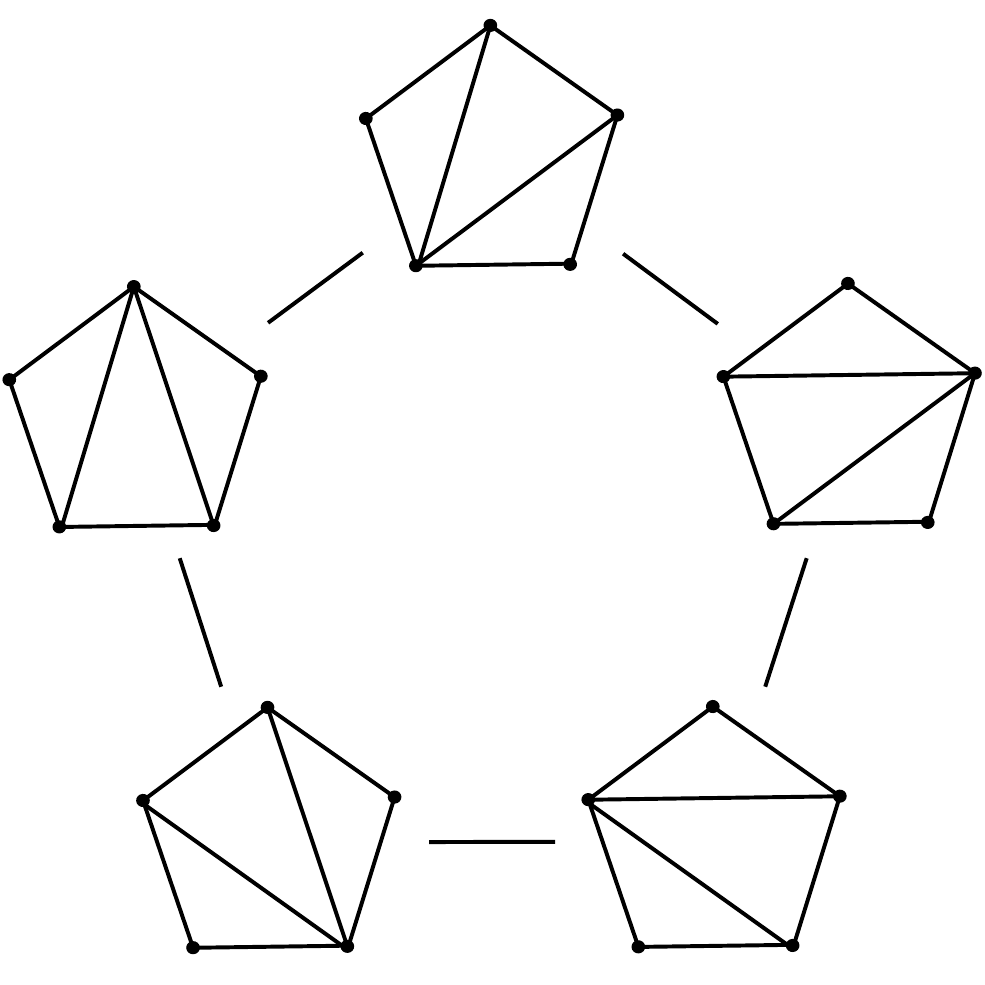}
    \caption{A 5-cycle in $\flip$}
    \label{figure: illustration (a) for the proof of lemma classification of paths of length 2}
\end{figure}

\begin{figure}[h]
    \labellist
        \pinlabel {$T$} at -5 58
        \pinlabel {$T_a$} at 34 105
        \pinlabel {$T_b$} at 34 6
        \pinlabel {$V$} at 85 86
        \pinlabel {$W$} at 85 24
    \endlabellist
    \centering
    \includegraphics{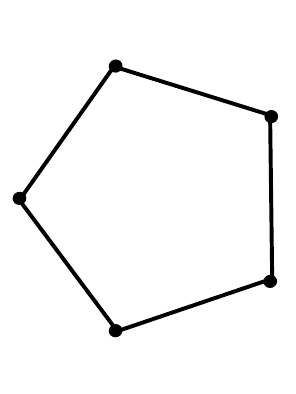}
    \caption{Illustration for the proof of Lemma \ref{lemma: classification of paths of length 2}}
    \label{figure: illustration (b) for the proof of lemma classification of paths of length 2}
\end{figure}

Suppose $a' = c'$ and $a\neq c$. Then $T_a = (T\backslash \{a\}) \cup \{c'\}$ implies $i(a,c') = 1$. However, $a, c'\in W$ implies $i(a,c') = 0$, which is a contradiction. Thus we conclude that $a = c$, i.e. $W = T_{ba}$. Similarly, we see that $V = T_{ab}$. Therefore $T_{ab}$ and $T_{ba}$ differ by a flip of $a'$ to $b'$ which means $i(a',b') = 1$. Hence $a$ and $b$ border exactly one common triangle in $T$ (see Figure \ref{figure: possible arrangements of two arcs in a triangulation}(ii)) and the 5-cycle containing $T$, $T_a$, and $T_b$ is unique. 
\end{proof}

The corollary below follows immediately from the uniqueness of the 4-cycle or 5-cycle containing a path of length two in $\flip$ from the proceeding lemma. 

\begin{corollary}\label{corollary: rigid 4-cycles and 5-cycles}
Let $\mathcal{Y}$ be a subgraph of $\flip$ and $\phi, \lambda:\mathcal{Y} \rightarrow \flipp$ two injective homomorphisms. If $\phi$ and $\lambda$ agree on three adjacent vertices of a 4-cycle or 5-cycle in $\mathcal{Y}$, then they agree on all vertices in the 4-cycle or 5-cycle.
\end{corollary}

\section{Illustrative Proof of Theorem \ref{theorem: main result} for Special Case}\label{section: proof of main theorem if d(S) = d(S')}

In Section \ref{section: proof of main theorem}, we give a complete proof of Theorem \ref{theorem: main result}. In this section, we provide a simpler proof for the case where $d(S) = d(S')$ and $S$ is non-exceptional, illustrating some of the key ideas. The proof in Section \ref{section: proof of main theorem} does not rely on the result in this section. 

\begin{proof}[Proof of Theorem \ref{theorem: main result} if $d(S)=d(S')$ and $S$ is non-exceptional.]
Fix $S$ and $S'$ with $d(S)=d(S')$. Fix a vertex $U\in \flip$. If for some $N\in\mathbb{N}$, there are no injective homomorphisms $\lambda:B_{\flip}(U;N)\rightarrow \flipp$, then $\mathcal{X}:= B_{\flip}(U;N)$ is a finite rigid set for $(S,S')$ trivially. Then suppose that for any $N\in\mathbb{N}$, an injective homomorphism $\lambda:B_{\flip}(U;N) \rightarrow \flipp$ does exist. Fix $r\geq 1$. Proposition \ref{prop: existence of weakly rigid set} says that there exists $R\geq r$ such that for any injective homomorphism $\lambda:B_{\flip}(U;R) \rightarrow \flipp$, there exists an injective homomorphism $\phi : \flip \rightarrow \flipp$ such that $\lambda|_{B_{\flip}(U;r)} = \phi|_{B_{\flip}(U;r)}$. We will prove that in fact, $\lambda|_{B_{\flip}(U;R)} = \phi|_{B_{\flip}(U;R)}$ and thus we can define our finite rigid set $\mathcal{X}$ to be $B_{\flip}(U;R)$.

It suffices to show that $\lambda|_{B_{\flip}(U;t)} = \phi|_{B_{\flip}(U;t)}$ implies $\lambda|_{B_{\flip}(U;t+1)} = \phi|_{B_{\flip}(U;t+1)}$ for $r\leq t \leq R-1$. Fix $V\in B_{\flip}(U;t+1)\backslash B_{\flip}(U;t)$ and assume $\lambda|_{B_{\flip}(U;t)} = \phi|_{B_{\flip}(U;t)}$. Our goal is to show that $\lambda(V) = \phi(V)$. There exists a path $\gamma$ of length $t+1$ from $U$ to $V$. Let $T$ be the vertex in $\gamma$ such that $d_{\flip}(U,T) = t$. Then there is an arc $a$ in $T$ such that $V = T_a$. We will henceforth refer to $V$ as $T_a$. Further, let $T_b$ refer to the vertex in $\gamma$ such that $d_{\flip}(U,T_b) = t-1$ and $b$ be the arc in $T$ which is flipped along to obtain $T_b$. See Figure \ref{figure: figure for attachment proof combined}(i). We consider three cases based on the relative positions of $a$ and $b$ in $T$.

\begin{figure}[h]
    \labellist
        \pinlabel {(i)} at 46 2 
        \pinlabel {$B_{\flip}(U;t)$} at 46 28 
        \pinlabel {$U$} at 46 65 
        \pinlabel {$T_b$} at 67 77 
        \pinlabel {$T$} at 86 85 
        \pinlabel {$T_a$} at 100 90 
        \pinlabel {(ii)} at 174 2
        \pinlabel {$B_{\flip}(U;t)$} at 174 28 
        \pinlabel {$U$} at 174 65 
        \pinlabel {$T_b$} at 195 77 
        \pinlabel {$T$} at 214 85 
        \pinlabel {$T_a$} at 228 90 
        \pinlabel {$W$} at 223 61 
        \pinlabel {(iii)} at 304 2
        \pinlabel {$B_{\flip}(U;t)$} at 304 28 
        \pinlabel {$U$} at 302 65 
        \pinlabel {$T_b$} at 323 77 
        \pinlabel {$T$} at 341 86 
        \pinlabel {$T_a$} at 356 90 
        \pinlabel {$T_{ba}$} at 354 59 
        \pinlabel {$T_{ab}$} at 367 71 
    \endlabellist
    \centering
    \includegraphics{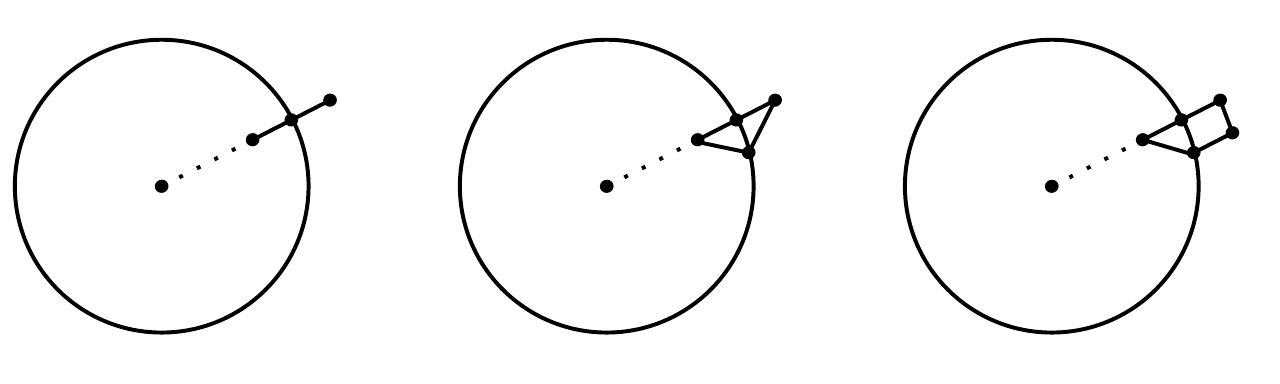}
    \caption{Arrangements of vertices near the boundary of a ball in $\flip$}
    \label{figure: figure for attachment proof combined}
\end{figure}
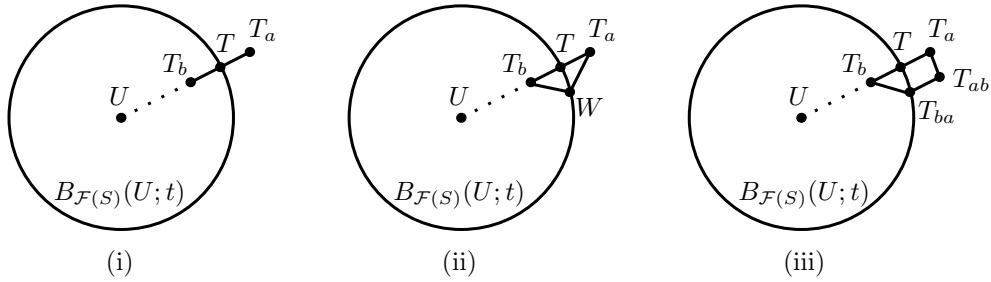

Case 1: The arcs $a$ and $b$ do not border any common triangles in $T$. Lemma \ref{lemma: classification of paths of length 2} tells us that $T_{ab} = T_{ba}$; call this triangulation $W$.  Then $T$, $T_a$, $T_b$, and $W$ form a 4-cycle. See Figure \ref{figure: figure for attachment proof combined}(ii). Since $T$, $T_b$, and $W$ are in $B_{\flip}(U;t)$, we see that $\lambda$ and $\phi$ agree on these vertices, and thus Corollary \ref{corollary: rigid 4-cycles and 5-cycles} implies that they also agree on $T_a$.

Case 2: The arcs $a$ and $b$ border exactly one common triangle in $T$. By Lemma \ref{lemma: classification of paths of length 2}, $T_{ab}$ and $T_{ba}$ are connected by an edge and the five vertices $T$, $T_a$, $T_b$, $T_{ab}$, and $T_{ba}$ form a 5-cycle. See Figure \ref{figure: figure for attachment proof combined}(iii). Since $T$, $T_b$, and $T_{ba}$ are in $B_{\flip}(U;t)$, we see that $\lambda$ and $\phi$ agree on these vertices, and hence Corollary \ref{corollary: rigid 4-cycles and 5-cycles} implies that they also agree on $T_a$ (and $T_{ab}$).

Case 3: The arcs $a$ and $b$ border exactly two common triangles. Then $T$ contains one of the following two arrangements shown in Figure \ref{figure: arrangement of arcs sharing two triangles}. Note that the arcs labelled $c$ and $d$ in the figure are not equal since $S$ is non-exceptional. Now by Theorem \ref{theorem: aramayona-koberda-parlier}, there is a homeomorphism $h:S \rightarrow S'$ which induces $\phi$, and consequently, $\phi$ is a graph isomorphism. This implies that $\text{deg}_{\flip}(T) = \text{deg}_{\flipp}(\phi(T))$ and since $\phi(T) = \lambda(T)$, it also follows that $\text{deg}_{\flip}(T) = \text{deg}_{\flipp}(\lambda(T))$. A flip of $T$ along any other arc $e\neq a,b$ will form a 4-cycle or a 5-cycle with $T$ and $T_b$ since $e$ can border at most one common triangle with $b$. By our arguments above, $\phi(T_e) = \lambda(T_e)$. Hence we conclude that $\phi$ and $\lambda$ agree on all of the vertices adjacent to $T$ in $\flip$ besides possibly $T_a$, and thus, they also agree on $T_a$.

Since $V=T_a$ was an arbitrary vertex in $B_{\flip}(U;t+1)\backslash B_{\flip}(U;t)$, we conclude that this holds for all such vertices, thus $\lambda|_{B_{\flip}(U;t+1)}=\phi|_{B_{\flip}(U;t+1)}$. Then by induction, it follows that $\lambda = \phi|_\mathcal{X}$. 

Now, say $\phi' : \flip \rightarrow \flipp$ is another injective homomorphism such that $\phi'|_{\mathcal{X}}  =  \lambda$. Then $\phi|_{\mathcal{X}}  =  \phi'|_{\mathcal{X}}$ and recall that we defined $\mathcal{X}  =  B_{R}(U)$. Using induction and the arguments above, we see that $\phi|_{B_{\flip}(U;i)}  =  \phi'|_{B_{\flip}(U;i)}$ for any $i\geq R$, hence $\phi  =  \phi'$. Thus $\phi$ is the unique injective homomorphism which extends $\lambda$. Then $\mathcal{X}$ is a finite rigid set for the pair $(S,S')$. 
\end{proof}

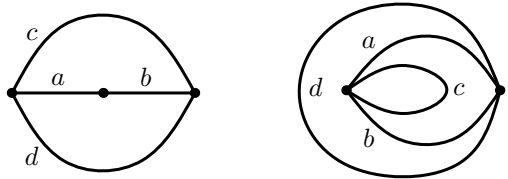
\begin{figure}[h]
    \labellist
        \pinlabel ${a}$ at 26 47 
        \pinlabel ${b}$ at 59 48 
        \pinlabel ${c}$ at 16 64 
        \pinlabel ${d}$ at 16 18 
        \pinlabel ${a}$ at 143 61 
        \pinlabel ${b}$ at 143 25 
        \pinlabel ${c}$ at 177 43 
        \pinlabel ${d}$ at 123 44 
    \endlabellist
    \centering
    \includegraphics{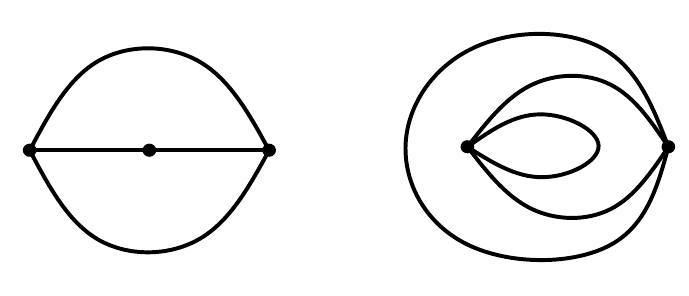}
    \caption{Arrangements of arcs where two arcs border two common triangles}
    \label{figure: arrangement of arcs sharing two triangles}
\end{figure}

Notice that this argument will not work if $d(S)<d(S')$ or if $S$ is exceptional. We give a more general proof that covers this case in the following section.

\section{Proofs of Theorems \ref{theorem: main result} and \ref{theorem: main result add-on}} \label{section: proof of main theorem}

Here we adjust our proof to avoid relying on the existence of a homeomorphism $S\rightarrow S'$ as we did in Section \ref{section: proof of main theorem if d(S) = d(S')}. This way, our result will hold more generally. Unfortunately, these adjustments demand a level of complexity to our surface. We call surfaces which do not have this required complexity \textit{simple}. They are $S  \cong  S_{0,n}$ with $1\leq n\leq 4$, $S_{0,0,(p_1)}$ with $p_1 \geq 1$, $S_{0,1,(p_1))}$ with $p_1\geq 1$, or one of $S_{0,2,(1)}$, $S_{0,2,(2)}$, $S_{0,0,(1,1)}$, $S_{0,0,(1,2)}$, $S_{0,0,(2,2)}$, $S_{1,1}$, and $S_{1,0,(1)}$. We call all other surfaces \emph{non-simple}. In this section, we first prove Theorem \ref{theorem: main result} for non-simple surfaces. Then we utilize different methods to prove it for the simple surfaces. Finally, we prove Theorem \ref{theorem: main result add-on}.

Assuming the surface is non-simple gives us the following property.

\begin{lemma}\label{lemma: nonsimple surfaces are complex enough}
Suppose a surface $S$ is non-simple and $T$ is a triangulation of $S$ containing arcs $a$, $b$, $c$, and $d$ in one of the configurations in Figure \ref{figure: arrangement of arcs sharing two triangles}. Then $c\neq d$ and one of $c$ and $d$ borders a second non-folded triangle in $T$ which has at least one side which is not equal to $c$ or $d$ and is not contained in $\partial S$. 
\end{lemma}

\begin{proof}
Suppose $c  =  d$. If $T$ contains the first arrangement in Figure \ref{figure: arrangement of arcs sharing two triangles}, then $S   =   S_{0,3}$, in which case $S$ is simple. If $T$ contains the second arrangement, then $S\cong S_{1,1}$, which is also simple. Thus we conclude that $c\neq d$. Suppose $c,d\subseteq \partial S$. If $T$ contains the first arrangement in Figure \ref{figure: arrangement of arcs sharing two triangles}, then $S \cong S_{0,1,(2))}$, in which case $S$ is simple. If $T$ contains the second arrangement, then $S\cong S_{0,0,(1,1)}$, which is also simple. So without loss of generality, assume $c\not \subseteq \partial S$. Hence $c$ borders a second triangle $\Delta$ in $T$. 

First suppose that $\Delta$ is folded. See Figure \ref{figure: complex proof 1}. If $d\subseteq \partial S$, then $S \cong S_{0,2,(1)}$, in which case $S$ is simple. Thus $d\not\subseteq \partial S$ and there is another triangle $\Delta'$ in $T$ with $d$ as a side. If $\Delta'$ is also folded, then $S \cong S_{0,4}$, a simple surface. We therefore may assume that $\Delta'$ is non-folded. Let $e$ and $f$ be the other sides of $\Delta'$. See Figure \ref{figure: complex proof 2}. If $e,f\subseteq \partial S$, then $S \not \cong S_{0,2,(2)}$, a simple surface. Thus one of $e$ and $f$ is not contained in $\partial S$ and we are done, since that side cannot be equal to $c$.

\begin{figure}[h]
    \labellist
        \pinlabel ${a}$ at 30 33 
        \pinlabel ${b}$ at 68 33 
        \pinlabel ${c}$ at 50 37 
        \pinlabel ${d}$ at 50 11 
        \pinlabel ${a}$ at 143 61 
        \pinlabel ${b}$ at 143 25 
        \pinlabel ${c}$ at 177 43 
        \pinlabel ${d}$ at 123 44 
    \endlabellist
    \centering
    \includegraphics{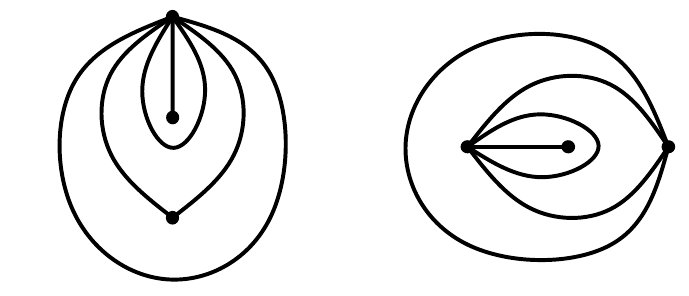}
    \caption{Triangles sharing two common sides with an adjacent folded triangle}
    \label{figure: complex proof 1}
\end{figure}
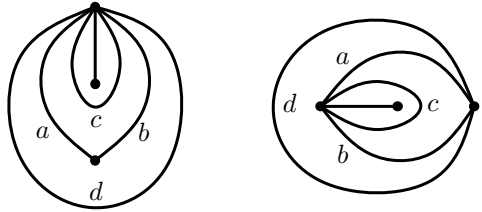

\begin{figure}[h]
    \labellist
        \pinlabel ${a}$ at 36 54 
        \pinlabel ${b}$ at 74 54 
        \pinlabel ${c}$ at 56 58 
        \pinlabel ${d}$ at 56 32 
        \pinlabel ${e}$ at 36 10 
        \pinlabel ${f}$ at 74 9 
        \pinlabel ${a}$ at 162.5 75 
        \pinlabel ${b}$ at 162.5 39 
        \pinlabel ${c}$ at 196.5 57 
        \pinlabel ${d}$ at 142.5 58 
        \pinlabel ${e}$ at 122 76 
        \pinlabel ${f}$ at 121 35 
    \endlabellist
    \centering
    \includegraphics{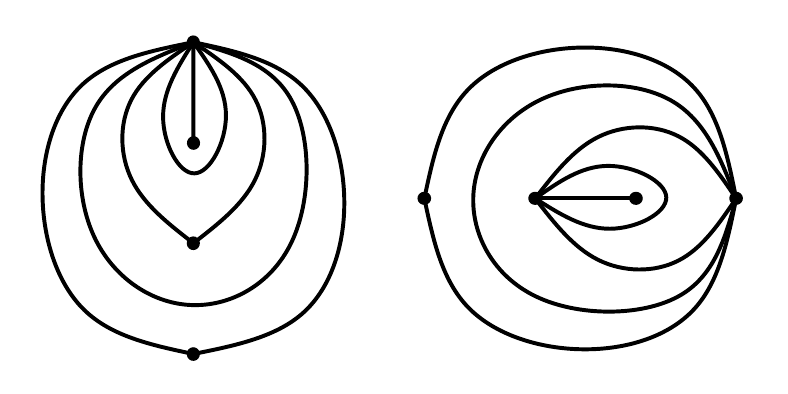}
    \caption{Triangles sharing two common sides with an adjacent folded triangle and an adjacent non-folded triangle}
    \label{figure: complex proof 2}
\end{figure}

Now suppose $\Delta$ is non-folded. Call its sides $e$ and $f$. See Figure \ref{figure: complex proof 3}. Suppose  $d,e,f\subseteq \partial S$. If $T$ contains the first arrangement in Figure \ref{figure: arrangement of arcs sharing two triangles}, then $S \cong S_{0,1,(3)}$, in which case $S$ is simple. If $T$ contains the second arrangement, then $S\cong S_{0,0,(1,2)}$, in which case $S$ is also simple. Thus one of $d$, $e$, and $f$ is not contained in $\partial S$. 

\begin{figure}[h]
    \labellist
        \pinlabel ${a}$ at 26 42 
        \pinlabel ${b}$ at 59 43 
        \pinlabel ${c}$ at 43 72 
        \pinlabel ${d}$ at 43 1 
        \pinlabel ${e}$ at 23 80 
        \pinlabel ${f}$ at 65 83 
        \pinlabel ${a}$ at 140 66 
        \pinlabel ${b}$ at 140 30 
        \pinlabel ${c}$ at 174 49 
        \pinlabel ${d}$ at 120 50 
        \pinlabel ${e}$ at 157 55 
        \pinlabel ${f}$ at 157 41.5 
    \endlabellist
    \centering
    \includegraphics{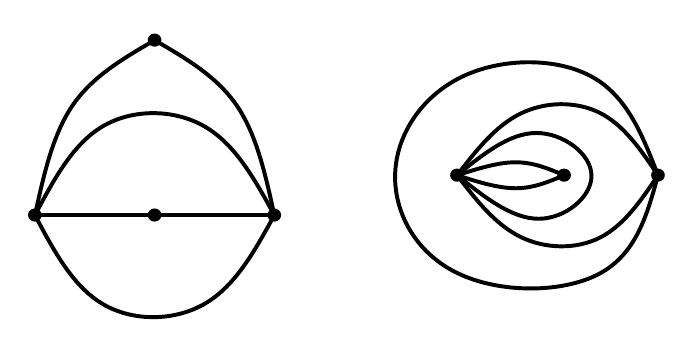}
    \caption{Triangles sharing two common sides with an adjacent non-folded triangle}
    \label{figure: complex proof 3}
\end{figure}
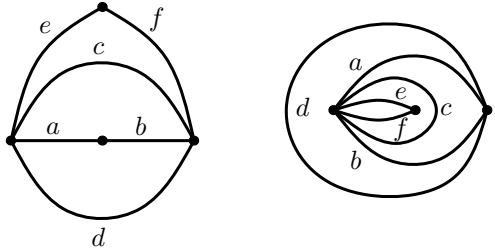

Suppose that either $e$ or $f$ is not contained in $\partial S$. Without loss of generality, say $e$ is not contained in $\partial S$. If $e\neq d$, we are done. If $e  =  d$, then $f\neq d$. If $f \not \subseteq \partial S$, we are done. So suppose $f \subseteq \partial S$. If $T$ contains the first arrangement in Figure \ref{figure: arrangement of arcs sharing two triangles}, then $S \cong S_{0,2,(1)}$, in which case $S$ is simple. If $T$ contains the second arrangement, then $S \cong S_{1,0,(1)}$, which is, once again, a simple surface. 

Finally, suppose $e,f\subseteq \partial S$ but $d\not \subseteq \partial S$. Thus there is another triangle $\Delta''$ with $d$ as a side. If $\Delta''$ is folded, then $S \cong S_{0,2,(2)}$, a simple surface. See Figure \ref{figure: complex proof 4}. Hence we conclude that $\Delta''$ is non-folded. Let $g$ and $h$ be the other sides of $\Delta''$. See Figure \ref{figure: complex proof 5}. If one of $g$ or $h$ is not contained in $\partial S$, we are done, since it cannot be equal to $c$. Suppose $g$ and $h$ are contained in $\partial S$. Thus if $T$ contains the first arrangement in Figure \ref{figure: arrangement of arcs sharing two triangles}, then $S \cong S_{0,1,(4)}$, in which case $S$ is simple. If $T$ contains the second arrangement, then $S \cong S_{0,0,(2,2)}$, in which case $S$ again is simple. This completes the proof.
\end{proof}

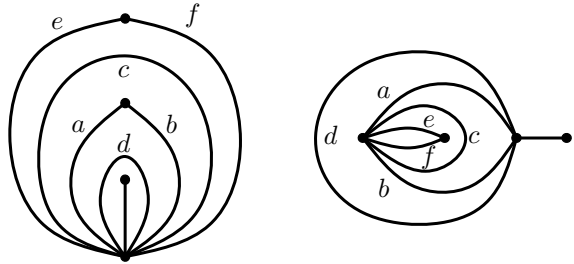
\begin{figure}[h]
    \labellist
        \pinlabel ${a}$ at 35 53 
        \pinlabel ${b}$ at 70 54 
        \pinlabel ${c}$ at 52 74 
        \pinlabel ${d}$ at 52 47 
        \pinlabel ${e}$ at 27 92 
        \pinlabel ${f}$ at 78 95 
        \pinlabel ${a}$ at 150 66 
        \pinlabel ${b}$ at 150 30 
        \pinlabel ${c}$ at 184 49 
        \pinlabel ${d}$ at 130 50 
        \pinlabel ${e}$ at 167 55 
        \pinlabel ${f}$ at 167 41.5 
    \endlabellist
    \centering
    \includegraphics{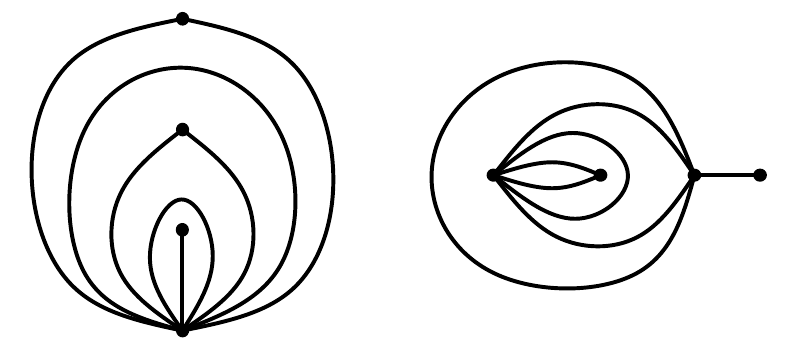}
    \caption{Triangles sharing two common sides with an adjacent folded triangle and an adjacent non-folded triangle}
    \label{figure: complex proof 4}
\end{figure}

\begin{figure}[h]
    \labellist
        \pinlabel ${a}$ at 26 62 
        \pinlabel ${b}$ at 59 63 
        \pinlabel ${c}$ at 43 92 
        \pinlabel ${d}$ at 43 21 
        \pinlabel ${e}$ at 23 100 
        \pinlabel ${f}$ at 65 103 
        \pinlabel ${g}$ at 23 12 
        \pinlabel ${h}$ at 65 12 
        \pinlabel ${a}$ at 140 74 
        \pinlabel ${b}$ at 140 38 
        \pinlabel ${c}$ at 174 57 
        \pinlabel ${d}$ at 120 58 
        \pinlabel ${e}$ at 157 63 
        \pinlabel ${f}$ at 157 49.25 
        \pinlabel ${g}$ at 110 93 
        \pinlabel ${h}$ at 110 19 
    \endlabellist
    \centering
    \includegraphics{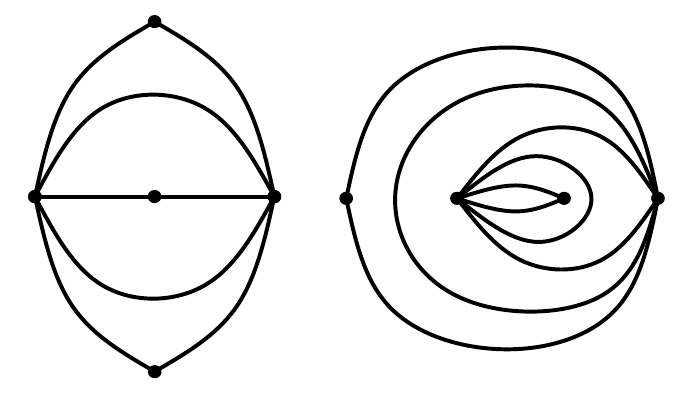} 
    \caption{Triangles sharing two common sides with two adjacent non-folded triangles}
    \label{figure: complex proof 5}
\end{figure}

Now we can prove our next lemma. We use the following notation: If $\mathcal{G}, \mathcal{H}$ are two subgraphs of $\flip$, then $\langle \mathcal{G}, \mathcal{H}\rangle$ refers to the induced subgraph of $\flip$ on all vertices of $\flip$ which are in $\mathcal{G}$ or $\mathcal{H}$.

\begin{lemma}\label{lemma: expansion of weakly rigid set}
Suppose $S$ is non-simple. Let $t\in \mathbb{N}$ and let $U$ and $V$ be vertices of $\flip$ such that 
\[d_{\flip}(U,V)  =  t+1\geq 2.\]
Then there exists a finite subgraph $\mathcal{B}\subseteq \flip$ containing $V$ such that for any injective homomorphisms $\lambda, \phi : \langle B_{\flip}(U;t), \mathcal{B}\rangle \rightarrow \flipp$, if $\lambda|_{B_{\flip}(U;t)}  =  \phi|_{B_{\flip}(U;t)}$, then $\lambda   =   \phi$. 
\end{lemma}

The first portion of the proof may remind the reader of parts of the proof in Section 4. Specifically, the case analysis, as well as the proofs of Case 1 and Case 2, are nearly identical. However, the latter portion of the proof is distinct, avoiding the assumptions $d(S)=d(S')$ and $S$ is non-exceptional, and instead appealing to Lemma \ref{lemma: nonsimple surfaces are complex enough}.

\begin{proof}[Proof of Lemma \ref{lemma: expansion of weakly rigid set}.]
Since $d_{\flip}(U,V)  =  t+1\geq 2$, there exists a path $\gamma$ of length $t+1$ from $U$ to $V$. Let $T$ be the vertex in $\gamma$ such that $d_{\flip}(U,T)  =  t$. Then there is an arc $a$ in $T$ such that $V  =  T_a$. We will henceforth refer to $V$ as $T_a$. Further, let $T_b$ refer to the vertex in $\gamma$ such that $d_{\flip}(U,T_b)  =  t-1$ and $b$ be the arc in $T$ which is flipped along to obtain $T_b$. See Figure \ref{figure: figure for attachment proof combined}(i). Then we consider three cases based on the relative positions of $a$ and $b$ in $T$.

Case 1: The arcs $a$ and $b$ do not border any common triangles in $T$. Lemma \ref{lemma: classification of paths of length 2} tells us that $T_{ab}  =  T_{ba}$; call this triangulation $W$. Then $T$, $T_a$, $T_b$, and $W$ form a 4-cycle. We define $\mathcal{B}  =  \{T_a\}$. See Figure \ref{figure: figure for attachment proof combined}(ii). If no injective homomorphism $\lambda: \langle B_{\flip}(U;t), \mathcal{B} \rangle \rightarrow \flipp$ exist, then the result holds trivially. Otherwise fix two injective homomorphisms $\lambda, \phi : \langle B_{\flip}(U;t), \mathcal{B}\rangle \rightarrow \flipp$ such that $\lambda_{B_{\flip}(U;t)}  =  \phi|_{B_{\flip}(U;t)}$.  Since $T$, $T_b$, and $W$ are in $B_{\flip}(U;t)$, we see that $\lambda$ and $\phi$ agree on these vertices, and then Corollary \ref{corollary: rigid 4-cycles and 5-cycles} implies that they also agree on $T_a$.

Case 2: The arcs $a$ and $b$ border exactly one common triangle in $T$. Lemma \ref{lemma: classification of paths of length 2} tells us that $T_{ab}$ and $T_{ba}$ are connected by an edge and the five vertices $T$, $T_a$, $T_b$, $T_{ab}$, and $T_{ba}$ form a 5-cycle. See Figure \ref{figure: figure for attachment proof combined}(iii). We define $\mathcal{B}  =  \{T_a$, $T_{ab}\}$. If no injective homomorphism $\lambda: \langle B_{\flip}(U;t), \mathcal{B} \rangle \rightarrow \flipp$ exist, then the result holds trivially. Otherwise fix two injective homomorphisms $\lambda, \phi : \langle B_{\flip}(U;t), \mathcal{B}\rangle \rightarrow \flipp$ such $\lambda|_{B_{\flip}(U;t)}  =  \phi|_{B_{\flip}(U;t)}$. Since $T$, $T_b$, and $T_{ba}$ are in $B_{\flip}(U;t)$, we see that $\lambda$ and $\phi$ agree on these vertices, and then Corollary \ref{corollary: rigid 4-cycles and 5-cycles} implies that they also agree on $T_a$ and $T_{ab}$.

Case 3: The arcs $a$ and $b$ border exactly two common triangles. Then $T$ contains one of the two arrangements shown in Figure \ref{figure: arrangement of arcs sharing two triangles}.

By Lemma \ref{lemma: nonsimple surfaces are complex enough}, we may assume that $c\neq d$ and that one of $c$ and $d$ borders a second non-folded triangle which has at least one side which is not equal to $c$ or $d$ and not contained in $\partial S$. Without loss of generality, assume that $c$ borders a second non-folded triangle with second side $e$ not part of $\partial S$ and not equal to $d$. Note that $e$ does not border a common triangle in $T$ with either $a$ or $b$. 

Then there is a 4-cycle in $\flip$ with vertices $T$, $T_a$, $T_e$, and $T_{ea}$ ($  =  T_{ae}$) and a 4-cycle in $\flip$ with vertices $T$, $T_b$, $T_e$, and $T_{eb}$ ($  =  T_{be}$). Further, in $\flip$, there is a 5-cycle with vertices $T$, $T_c$, $T_a$, $T_{ac}$ and $T_{ca}$, a 5-cycle with vertices $T$, $T_c$, $T_b$, $T_{bc}$ and $T_{cb}$, and a 5-cycle with vertices $T$, $T_c$, $T_e$, $T_{ec}$ and $T_{ce}$. We define $\mathcal{B}$ to be the set containing the vertices of these two 4-cycles and three 5-cycles. See Figure \ref{figure: expansion proof}. 

\begin{figure}[h]
    \labellist
        \pinlabel {$B_{\flip}(U;t)$} at 58 20
        \pinlabel {$U$} at 58 71
        \pinlabel {$T_b$} at 85 87
        \pinlabel {$T$} at 107 102
        \pinlabel {$T_a$} at 132.5 107
        \pinlabel {$T_{bc}$} at 87.5 119 
        \pinlabel {$T_{cb}$} at 101 137.5
        \pinlabel {$T_c$} at 115.5 129
        \pinlabel {$T_{ce}$} at 127 150
        \pinlabel {$T_e$} at 140 80.5
        \pinlabel {$T_{ec}$} at 171 60
        \pinlabel {$T_{ea}$} at 166.5 89 
        \pinlabel {$T_{be}$} at 123 58 
        \pinlabel {$T_{ac}$} at 158.5 115
        \pinlabel {$T_{ca}$} at 142 134
    \endlabellist
    \centering
    \includegraphics{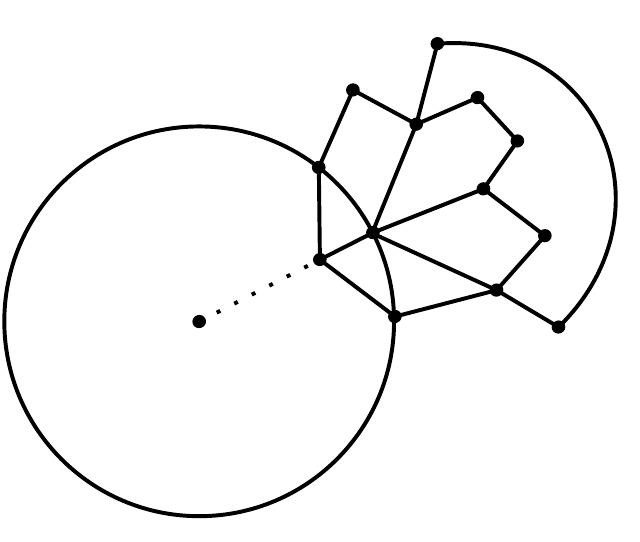}
    \caption{Positions of vertices in $\flip$ from the proof of Lemma \ref{lemma: expansion of weakly rigid set}}
    \label{figure: expansion proof}
\end{figure}

If no injective homomorphism $\lambda: \langle B_{\flip}(U;t), \mathcal{B} \rangle \rightarrow \flipp$ exist, then the result holds trivially. Otherwise fix two injective homomorphisms $\lambda, \phi : \langle B_{\flip}(U;t), \mathcal{B}\rangle \rightarrow \flipp$ such that $\lambda|_{B_{\flip}(U;t)} = \phi|_{B_{\flip}(U;t)}$. Corollary \ref{corollary: rigid 4-cycles and 5-cycles} implies that $\phi$ and $\lambda$ also agree on $T_c$ and $T_e$. Let $b'$ be the arc flipped along from $\lambda(T) = \phi(T)$ to $\lambda(T_b) = \phi(T_b)$, let $c'$ be the arc flipped along from $\lambda(T) = \phi(T)$ to $\lambda(T_c) = \phi(T_c)$, and let $e'$ be the arc flipped along from $\lambda(T) = \phi(T)$ to $\lambda(T_e) = \phi(T_e)$. Let $a'$ be the arc flipped along from $\lambda(T)$ to $\lambda(T_a)$ and $a^*$ the arc flipped along from $\phi(T)$ to $\phi(T_a)$. It suffices to show $a' = a^*$. For this will imply that $\lambda$ and $\phi$ agree on $T_a$ and then further applications of Corollary \ref{corollary: rigid 4-cycles and 5-cycles} will show that they also agree on all of $\mathcal{B}$. 

Applying Lemma \ref{lemma: classification of paths of length 2}, we see that $b'$ and $c'$ border a common triangle, as do $c'$ and $e'$ and that $b'$ and $e'$ do not border a common triangle. See Figure \ref{figure: arc arrangements for expansion proof} for the possible arrangements of $b'$, $c'$, and $e'$ in $S'$. We also see that $a'$ borders a common triangle with $c'$ and not with $e'$. The same is true for $a^*$. Then there is only one possibility for $a'$ and $a^*$, thus they must be equal. By Corollary \ref{corollary: rigid 4-cycles and 5-cycles}, it follows that $\phi$ and $\lambda$ agree on all of $\mathcal{B}$. 
\end{proof}

\begin{figure}[h]
    \labellist
        \pinlabel {$c'$} at 42 49 
        \pinlabel {$c'$} at 154 49 
        \pinlabel {$b'$} at 67 67 
        \pinlabel {$b'$} at 133 67 
        \pinlabel {$e'$} at 64 17 
        \pinlabel {$e'$} at 176 17 
    \endlabellist
    \centering
    \includegraphics{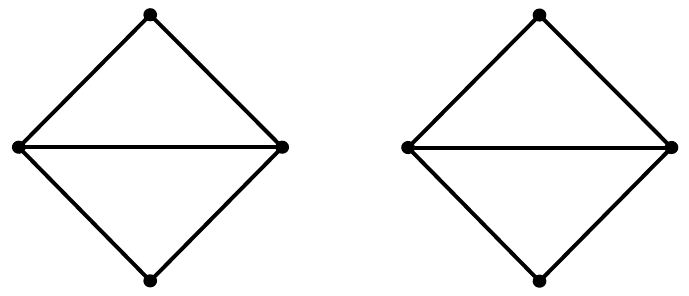}
    \caption{Arc arrangements from the proof of Lemma \ref{lemma: expansion of weakly rigid set}}
    \label{figure: arc arrangements for expansion proof}
\end{figure}

We are now ready to prove Theorem \ref{theorem: main result} for the case when $S$ is non-simple.

\begin{proof}[Proof of Theorem \ref{theorem: main result} if $S$ is non-simple.]
Fix $S$ and $S'$. Fix a vertex $U\in \flip$. If for some $N\in\mathbb{N}$, there are no injective homomorphisms $\lambda:B_{\flip}(U;N)\rightarrow \flipp$, then $\mathcal{X}:= B_{\flip}(U;N)$ is a finite rigid set for $(S,S')$ trivially. Then assume that for any $N\in\mathbb{N}$, an injective homomorphism $\lambda:B_{\flip}(U;N) \rightarrow \flipp$ does exist. Fix $r\geq 1$. Apply Proposition \ref{prop: existence of weakly rigid set} to obtain $R\geq r$ such that for any injective homomorphism $\lambda:B_{\flip}(U;R)\rightarrow \flipp$, there is an injective homomorphism $\phi : \flip \rightarrow \flipp$ such that $\lambda|_{B_{\flip}(U;r)} = \phi|_{B_{\flip}(U;r)}$.

For each $r< j \leq R$, let $V_{1}^j, V_{2}^j, \ldots, V_{k_j}^j$ be the vertices of $\flip$ such that $d_{\flip}(U, V_i^j) = j$. Recall that since $\flip$ has finite valence, there are finitely many such vertices for each $j$. Let $\mathcal{B}_{i}^j$ be the finite subgraph from Lemma \ref{lemma: expansion of weakly rigid set} for $t = j-1$, $U = U$, and $V = V_{i}^j$ for each $r < j \leq R$ and $1 \leq i \leq k_j$. Define
\[\mathcal{X} = \left\langle B_{\flip}(U;R), \bigcup_{r<j\leq R} \bigcup_{1\leq i \leq k_j} \mathcal{B}_{i}^j\right\rangle.\]
Let $\lambda:\mathcal{X}\rightarrow \flipp$ be an injective homomorphism. By Proposition \ref{prop: existence of weakly rigid set}, there exists an injective homomorphism $\phi : \flip \rightarrow \flipp$ such that $\lambda|_{B_{\flip}(U;r)} = \phi|_{B_{\flip}(U;r)}$. For $r< j \leq R$, Lemma \ref{lemma: expansion of weakly rigid set} tells us that if $\lambda$ and $\phi$ agree on $B_{\flip}(U;j-1)$, then they also agree on $\left\langle B_{\flip}(U;j-1), \bigcup_{1\leq i \leq k_{j}} \mathcal{B}_{i}^{j} \right \rangle \supseteq B_{\flip}(U;j)$. By induction, we conclude that $\lambda$ and $\phi$ agree on all of $\mathcal{X}$, the domain of $\lambda$, i.e. $\lambda = \phi|_{\mathcal{X}}$.

Say $\phi' :\flip \rightarrow \flipp$ is another injective homomorphism such that $\lambda = \phi' |_{\mathcal{X}}$. Then in particular, $\phi$ and $\phi'$ agree on $B_{\flip}(U;R)\subseteq \mathcal{X}$. Using induction and the arguments above, we see that $\phi|_{B_{\flip}(U;i)}  =  \phi'|_{B_{\flip}(U;i)}$ for any $i\geq R$, hence $\phi  =  \phi'$. Thus $\phi$ is the unique injective homomorphism which extends $\lambda$. Then $\mathcal{X}$ is a finite rigid set for the pair $(S,S')$. 
\end{proof}

For simple surfaces, we use explicit descriptions of the flip graphs of these surfaces to prove Theorems \ref{theorem: main result} and \ref{theorem: main result add-on}. In the case that $S\cong S_{0,n}$ for $1\leq n \leq 3$, $S_{0,0,(p_1)}$ for $p_1 \leq 1$, or $S_{0,1,(p_1))}$ for $p_1 \leq 1$, we observe that $\flip$ is finite. If $S  \cong  S_{0,0,(1,1)}$, then $\flip$ is a bi-infinite line. If $S  \cong  S_{1,1}$, then $\flip$ is a trivalent tree. For the six remaining simple surfaces, we include an image of a portion of $\flip$ in Appendix \ref{appendix: images of flip graphs}. The images were created by recognizing symmetries of $\flip$ induced by homeomorphisms of $S$ and manually drawing a sufficiently large subgraph of $\flip$ so that the union of the images of this subgraph under the symmetries yields all of $\flip$. 

\begin{proof} [Proof of Theorem \ref{theorem: main result} if $S$ is simple.]

First suppose that $S\cong S_{0,n}$ for $1\leq n \leq 3$, $S_{0,0,(p_1)}$ for $p_1 \leq 1$, or $S_{0,1,(p_1))}$ for $p_1 \leq 1$. Then $\flip$ is finite, we can set $\mathcal{X} = \flip$, and the result trivially follows. Recall that Theorem \ref{theorem: main result} does not apply if $S\cong S_{0,0,(1,1)}$ or $S_{1,1}$.

Now suppose $S$ is one of the six remaining simple surfaces. For each of these surfaces, we appeal to images of $\flip$ in Appendix \ref{appendix: images of flip graphs} to make our argument. We demonstrate the argument for $S_{0,0,(1,2)}$ below. The argument for the other five surfaces is similar. 

The image of $\flip$ for when $S  =  S_{0,0,(1,2)}$ is shown in Figure \ref{figure: flip graph simple case 1} as well as in Appendix \ref{appendix: images of flip graphs}. Fix a vertex $U\in \flip$ which is a member of three distinct 5-cycles, as pictured in the figure. Fix $S'$ and apply Proposition \ref{prop: existence of weakly rigid set} for $X=\flip$, $Y=\flipp$, fixed $r\geq 1$, and $x=U$ to obtain a value for $R$. Then we define $\mathcal{X}$ as the union of $B_{\flip}(U;R)$ together with any edges and vertices forming a 5-cycle with vertices in $B_{\flip}(U;R)$. Observe that $\mathcal{X}$ is finite and connected.

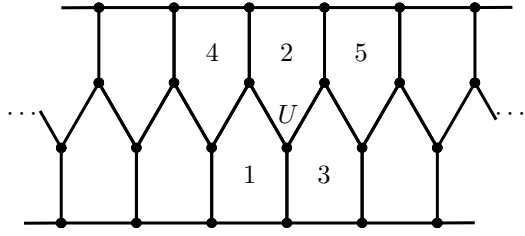
\begin{figure}[h]
    \labellist
        \pinlabel {$\cdots$} at 0 46
        \pinlabel {$\cdots$} at 184 46
        \pinlabel {$1$} at 85 23
        \pinlabel {$3$} at 113 23
        \pinlabel {$4$} at 71 69
        \pinlabel {$2$} at 99 69
        \pinlabel {$5$} at 127 69
        \pinlabel {$U$} at 99.5 46
    \endlabellist
    \centering
    \includegraphics{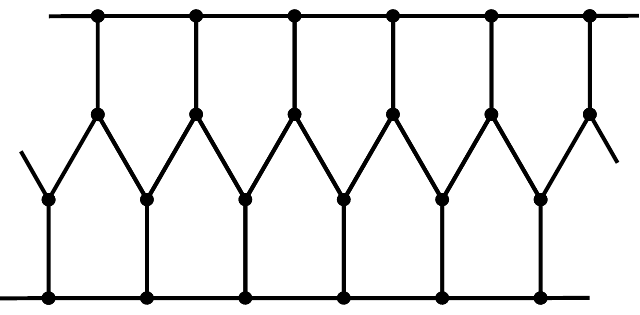}
    \caption{Flip graph of $S_{0,0,(1,2)}$}
    \label{figure: flip graph simple case 1}
\end{figure}

Let $\lambda: \mathcal{X} \rightarrow \flipp$ be an injective homomorphism. From Proposition \ref{prop: existence of weakly rigid set}, there exists an injective homomorphism $\phi : \flip \rightarrow \flipp$ such that $\lambda|_{B_{\flip}(U;r)} = \phi|_{B_{\flip}(U;r)}$. We wish to show that $\lambda = \phi|_{\mathcal{X}}$, and thus that $\mathcal{X}$ is a finite rigid set for the pair $(S,S')$. 

Observe that since $r\geq 1$ that $B_{\flip}(U;r)$ contains three adjacent vertices in each of the 5-cycles labeled 1, 2, and 3 in Figure \ref{figure: flip graph simple case 1}. By assumption, $\lambda$ and $\phi$ agree on these vertices. By Corollary \ref{corollary: rigid 4-cycles and 5-cycles}, it follows that $\lambda$ and $\phi$ agree on all the vertices in these 5-cycles. Then $\lambda$ and $\phi$ agree on three adjacent vertices of the 5-cycles labeled 4 and 5 in Figure \ref{figure: flip graph simple case 1} and we apply Corollary \ref{corollary: rigid 4-cycles and 5-cycles} again. Repeated applications show that in fact, $\lambda$ and $\phi$ agree on all of $\mathcal{X}$. 

If $\phi': \flip \rightarrow \flipp$ is another injective homomorphism which agrees with $\lambda$ on $\mathcal{X}$, then $\phi$ and $\phi'$ agree on $\mathcal{X}$. In the same way as above, we can repeatedly apply Corollary \ref{corollary: rigid 4-cycles and 5-cycles} to show that indeed $\phi=\phi'$, thus there is a unique extension of $\lambda$.

For each remaining simple surface $S$, we similarly observe that if two injective homomorphisms agree on $B_{\flip}(U;1)$ for the vertex marked $U$ on $\flip$ in Appendix \ref{appendix: images of flip graphs}, we can apply Corollary \ref{corollary: rigid 4-cycles and 5-cycles} to find more vertices on which the two maps agree, which gives more vertices upon which to apply the corollary. Using the images in the appendix, we observe that repeated applications of the corollary to these graphs show that the injective homomorphisms agree on arbitrarily large balls, provided the domains of the maps are sufficiently large and that in fact, the maps agree on all vertices required to show that they agree on the ball. Then we apply Proposition \ref{prop: existence of weakly rigid set} with $X=\flip$, $Y=\flipp$, $r\geq 1$, and $x=U$ to find a sufficient $R\geq r$. We define $\mathcal{X}$ with $B_{\flip}(U;R)\subseteq\mathcal{X}\subseteq \flip$ in a way so that by repeatedly applying Corollary \ref{corollary: rigid 4-cycles and 5-cycles}, we are able to show that for any injective homomorphism $\lambda:\mathcal{X}\rightarrow \flipp$, there is an injective homomorphism $\phi: \flip \rightarrow \flipp$ which agrees with $\lambda$ not only on $B_{\flip}(U;r)$, but on all of $\mathcal{X}$. As before, if there is another injective homomorphism $\phi':\flip \rightarrow \flipp$ which agrees with $\lambda$ on $\mathcal{X}$, repeated applications of Corollary \ref{corollary: rigid 4-cycles and 5-cycles} show that $\phi$ and $\phi'$ are the same. We conclude that $\lambda$ has a unique extension, hence $\mathcal{X}$ is a finite rigid set. 
\end{proof}

Finally, we prove Theorem \ref{theorem: main result add-on}.

\begin{proof}[Proof of Theorem \ref{theorem: main result add-on}]
We assume that $S \cong S_{0,0,(1,1)}$. Recall that $\flip$ is a bi-infinite line. First suppose $S'$ is a surface such that $\flipp$ is finite. Say $\flipp$ has $k$ vertices. Choose a finite subgraph $\mathcal{X}$ of $\flip$ with at least $k+1$ vertices. Then $\mathcal{X}$ is a finite rigid set for the pair $(S,S')$ trivially, since there are no injective homomorphisms $\mathcal{X}\rightarrow \flipp$. 

Now, suppose $S' \cong S \cong  S_{0,0,(1,1)}$. Let $\mathcal{X}$ be any finite connected segment of $\flip$ of length at least 1. Observe that $\mathcal{X}$ is a finite rigid set for the pair $(S,S')$. 

Finally, suppose $\flipp$ is infinite and $S' \not \cong S_{0,0,(1,1)}$. We will show that no finite rigid set exists for $(S,S')$. Let $\mathcal{Y}$ be a finite subgraph of $\flip$. Suppose $\mathcal{Y}$ is not connected and has two connected components of distance $d$ in $\flip$. Then it is easy to construct an injective homomorphism $\lambda:\mathcal{Y}\rightarrow \flipp$ where the images of these connected components have distance greater than $d$. Thus $\lambda$ cannot be extended to an injective homomorphism $\flip \rightarrow \flipp$. 

Now suppose $\mathcal{Y}$ is connected. Then $\mathcal{Y}$ is a line segment of some finite length $l$. Suppose $S'$ is simple, in addition to the already present assumptions about $S'$. Then either $\flipp$ is a trivalent tree (if $S' \cong S_{1,1}$) or $\flipp$ is pictured in Appendix \ref{appendix: images of flip graphs}. In each case, we can construct an injective homomorphism of $\mathcal{Y}$ into $\flipp$ which admits multiple distinct extensions to all of $\flip$.

Now instead, suppose $S'$ is non-simple. By inspection, we see that any non-simple $S'$ admits a triangulation $T$ with arcs in the second arrangement in Figure \ref{figure: arrangement of arcs sharing two triangles}. Let $a$, $b$, $c$, and $d$ refer to the arcs in $T$ corresponding to the similarly labeled arcs in the figure. By Lemma 5.1, there is an arc $e$ in $T$ which is not equal to $a$, $b$, $c$, or $d$ and can be flipped along. We can generate a bi-infinite line $\gamma$ in $\flip$ in which every triangulation contains $T\backslash \{a,b\}$. We can also generate a bi-infinite line $\gamma'$ in which every triangulation contains $T_e \backslash \{a,b\}$. These two lines are disjoint and every point on each line is connected to a point on the other line, so that $\gamma$, $\gamma'$ and the edges connecting them form a shape resembling a bi-infinite ladder. See Figure \ref{figure: bi-infinite ladder}. Then there exists an injective homomorphism $\lambda: \mathcal{Y} \rightarrow \gamma \subseteq \flipp$ which has distinct extensions $\phi,\phi' : \flip \rightarrow \flipp$ where $\phi(\flip)\subseteq \gamma$ and $\phi'$ maps part of $\flip\backslash \mathcal{Y}$ onto $\gamma'$. Thus, since extensions of injective homomorphisms of $\mathcal{Y}$ into $\flipp$ are not always unique, $\mathcal{Y}$ is not a finite rigid set for $(S,S')$. Then indeed, no finite rigid set exists for $(S,S')$.
\end{proof}

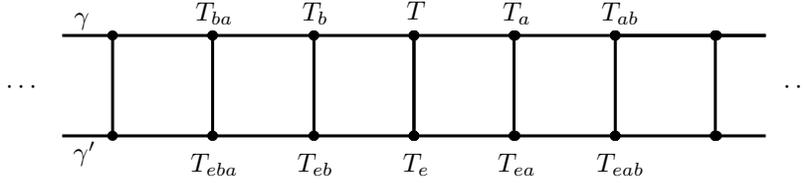
\begin{figure}[h]
    \labellist
        \pinlabel {$T$} at 155 70
        \pinlabel {$T_e$} at 155 12
        \pinlabel {$T_a$} at 193 69
        \pinlabel {$T_{ea}$} at 193 12
        \pinlabel {$T_{ab}$} at 232 69
        \pinlabel {$T_{eab}$} at 232 12
        \pinlabel {$T_b$} at 117 69
        \pinlabel {$T_{eb}$} at 117 12
        \pinlabel {$T_{ba}$} at 79 69
        \pinlabel {$T_{eba}$} at 79 12
        \pinlabel {$\gamma$} at 29 66
        \pinlabel {$\gamma '$} at 30 16
        \pinlabel {$\cdots$} at 7 41
        \pinlabel {$\cdots$} at 300 41
    \endlabellist
    \centering
    \includegraphics{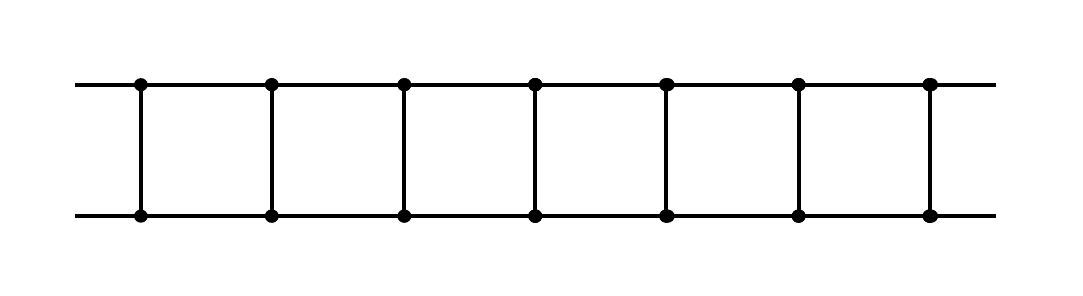}
    \caption{Bi-infinite ladder structure in the flip graph}
    \label{figure: bi-infinite ladder}
\end{figure}

\section{Proofs of Corollaries \ref{corollary: applying a-k-p theorem} and \ref{corollary: exhaustion}}\label{section: proof of corollaries}

In all cases in the previous section where we proved the existence of a finite rigid set for the pair $(S,S')$, the set could be chosen to be arbitrarily large if $\flip$ was infinite or all of to be $\flip$ if $\flip$ was finite. Thus, Corollary \ref{corollary: exhaustion} follows immediately.

Theorems \ref{theorem: main result} and \ref{theorem: aramayona-koberda-parlier} nearly gives us Corollary \ref{corollary: applying a-k-p theorem}, but it remains to show the following. 

\begin{lemma} \label{lemma: akp uniqueness}
In Theorem \ref{theorem: aramayona-koberda-parlier}, the embedding $h$ is unique up to isotopy.
\end{lemma}

\begin{proof} Fix non-exceptional $S$ and fix $S'$. If there is no injective homomorphism $\phi: \flip \rightarrow \flipp$, the result is trivial, so assume an injective homomorphism does exist. Suppose $h,h':S\rightarrow S'$ are two embeddings which induce $\phi$. Let $Q$ and $Q'$ be the collections of arcs on $S'$ such that $h(T) \cup Q = \phi(T) = h'(T)\cup Q'$ for any triangulation $T$ of $S$. 

Fix a triangulation $U$ of $S$. If $a\in U$ is a flippable arc, then $h(U)\cup Q = h'(U)\cup Q'$ and $h(U_a)\cup Q = h'(U_a)\cup Q'$. From this, we can conclude that $h$ and $h'$ map $a$ to the same arc on $S'$. 

Fix a triangle $\Delta$ in $U$. Suppose $\Delta$ is non-folded and all three of its sides are flippable arcs. Then $h$ and $h'$ agree on which arc they map each side to. $S'$ cannot be isotopic to $S_{0,3}$ or $S_{1,1}$ because $S$ is non-exceptional and embeds into $S$. Then there is only one triangle on $S'$ bordered by the images of these three sides. It follows that, up to isotopy, $h$ and $h'$ agree on $\Delta$.

Now suppose $\Delta$ is non-folded and exactly two of its sides are flippable arcs, hence the third side is part of $\partial S$. Call the flippable sides $a$ and $b$. Then $h$ and $h'$ agree on which arcs they map $a$ and $b$ to. Call these arcs $a'$ and $b'$, respectively. If $a'$ and $b'$ border only one common triangle in $h(U)\cup Q = h'(U)\cup Q'$, then $h$ and $h'$ must both map $\Delta$ to it, and in such a way so that $h$ and $h'$ are isotopic on $\Delta$. Now suppose $a'$ and $b'$ border two common triangles. Since $a$ and $b$ are flippable in $U$, each border a second triangle in $U$, and these triangles must be the same, since $h$ and $h'$ are embeddings. Call this second triangle $\Delta'$. Since $S$ is not isotopic to $S_{0,1,(2))}$, the third side of $\Delta'$ is a flippable arc. By our arguments above, we conclude that up to isotopy, $h$ and $h'$ agree on $\Delta'$ and thus on $\Delta$ as well.

Now suppose $\Delta$ is non-folded and exactly one of its sides is a flippable arc, hence its other two sides are part of $\partial S$. Call the flippable side $a$. As before, $h$ and $h'$ send $a$ to the same arc on $S'$. Call this arc $a'$. Since $a$ is a flippable arc, there is another non-folded triangle $\Delta'$ in $U$ with $a$ as a side. Further, since $S$ is not isotopic to $S_{0,0,(4)}$, a second side of $\Delta'$ is a flippable arc. By our arguments above, we conclude that up to isotopy, $h$ and $h'$ agree on $\Delta'$. Since $a'$ can border at most two triangles, the maps $h$ and $h'$ 
agree on $\Delta$, up to isotopy, as well. 

Finally, suppose $\Delta$ is a folded triangle. Call its outer arc $a$. Since $S$ is not isotopic to $S_{0,1,(1))}$, $a$ is a flippable arc. Thus, there is another triangle $\Delta'$ in $U$ with $a$ as a side. Since $S$ is not isotopic to $S_{0,3}$, $\Delta'$ is a non-folded triangle. By our arguments above, we conclude that up to isotopy, $h$ and $h'$ agree on $\Delta'$, hence they agree on $\Delta$, as well.

We conclude that $h$ and $h'$ are isotopic. 
\end{proof}

\section{Regarding the Remaining Case} \label{section: remaining case}

The flip graph of $S \cong S_{1,1}$ is a trivalent tree and all of our above methods above fail in this case. We discuss in this section some partial results.

For a surface $S'$, if balls in $\flipp$ grow more slowly than balls in $\flip$, we can use cardinality arguments to determine whether a finite rigid set for $(S,S')$ exists.

\begin{lemma}\label{lemma: remaining case works if s' balls grow slowly}
Let $S \cong S_{1,1}$. Suppose $S'$ is a surface such that there exists $n\in \N$, such that for all vertices $V$ in $\flipp$, the closed ball $B_{\flipp}(V;n)$ has fewer than $3\cdot 2^n -2$ vertices. Then for any vertex $U$ in $\flip$, the set $B_\flip(U;n)$ is a finite rigid set for the pair $(S,S')$. 
\end{lemma}

\begin{proof}
Fix $n\in \N$ such that for all vertices $V$ in $\flipp$, the closed ball $B_{\flipp}(V;n)$ has fewer than $3\cdot 2^n -2$ vertices. Fix a vertex $U$ in $\flip$ and define $\mathcal{X} = B_{\flip}(U;n)$. Since $\flip$ is a trivalent tree, the number of vertices in $\mathcal{X}$ is $3\cdot 2^n-2$. Then there are no injective homomorphisms $\mathcal{X} \rightarrow \flipp$, hence $\mathcal{X}$ is a finite rigid set trivially.
\end{proof}

We can apply the proceeding lemma to a class of surfaces $S'$ where $\text{Mod}(S')$ has sub-exponential growth. 

\begin{prop}\label{prop: remaining case subexponential growth}
Let $S \cong S_{1,1}$ and let $S'$ be a sphere with $b + n \leq 3$. Then there is a finite rigid set for the pair $(S,S')$. 
\end{prop}

\begin{proof}
The flip graph $\flipp$ is quasi-isometric to $\text{Mod}(S')$ (e.g. see Lemma 2.3 in \cite{disarlo2019geometry}). By inspection, we can see that if $S'$ is a sphere with $b+n \leq 3$, then $\text{Mod}(S')$ has sub-exponential growth, hence $\flipp$ does as well. Recall that there are finitely many $
\text{Aut}(\flipp)$-orbits of vertices in $\flipp$, so to find the maximum size of a ball of radius $n$ in $\flipp$, we only need to consider the size of a ball of radius $n$ of a member of each of the finitely many orbits. Then since $\flipp$ has sub-exponential growth, we conclude that there exists $n\in \N$, such that for all vertices $V$ in $\flipp$, the closed ball $B_{\flipp}(V;n)$ has fewer than $3\cdot 2^n -2$ vertices. Then we may apply Lemma \ref{lemma: remaining case works if s' balls grow slowly} to get the result.
\end{proof}

Unfortunately for us, if $S'$ is such that $g\geq 1$ or $b+n \geq 4$, then $\text{Mod}(S')$ has exponential growth. Indeed, in this case, $S'$ admits two essential simple closed curves with intersection at least one and squares of Dehn twists around these curves generate a rank 2 free group within $\text{Mod}(S')$. Thus, we cannot use an argument like the one for Proposition \ref{prop: remaining case subexponential growth} in this case. Further, among flip graphs with exponential growth rate, it is not known when this growth rate is less than that of $2^n$, so Lemma \ref{lemma: remaining case works if s' balls grow slowly} is not presently helpful either. In \cite{rafi2016uniform}, Rafi and Tao proved a uniform upper bound on the growth rate of flip graphs, but it is greater than $2^n$. We do not know at this time if the bound could be improved for some or all of the surfaces with exponential flip graph growth. 

On the other hand, however, if we can determine that $\flipp$ contains a trivalent tree, we can conclude the following:

\begin{lemma}\label{lemma: remaining case containing trivalent tree means no finite rigid sets}
Let $S \cong S_{1,1}$ and let $S'$ be such that $\flipp$ contains a trivalent tree. Then there are no finite rigid sets for the pair $(S,S')$.
\end{lemma}

\begin{proof}
There are infinitely many distinct injective homomorphisms of $\flip$ onto the trivalent tree in $\flipp$, even with the images of any finite set of vertices of $\flip$ fixed. We require injective homomorphisms of our finite rigid sets to have \textit{unique} extension to all of $\flip$, so there are no finite rigid sets in this case. 
\end{proof}

Since $\mathcal{F}(S_{1,1})$ is a trivalent tree, we get the following corollary.

\begin{corollary}
If $S \cong S' \cong S_{1,1}$, then there are no finite rigid sets for the pair $(S,S')$.
\end{corollary}

Unfortunately, we do not know at this time if any other flip graphs contain a trivalent tree. Further, growth rate at least that of $2^n$ is not a sufficient condition to imply the existence of a trivalent tree, so we cannot at this time reduce the problem to a question solely about growth rate.  

Finally, we leave the reader with an example of a flip graph we have found which does not contain a trivalent tree despite having exponential growth. In fact, we show the even stronger result that the flip graph does not contain a 62-ball of any vertex in the trivalent tree.

\begin{prop}\label{prop: remaining case specific example}
Let $S \cong S_{1,1}$ and $S' \cong S_{1,0,(1)}$. For any vertex $U$ in $\flip$, there does not exist an injective homomorphism $B_{\flip}(U;62) \rightarrow \flipp$. 
\end{prop}

Then our original question is answered in this case:

\begin{corollary}
Let $S \cong S_{1,1}$ and $S' \cong S_{1,0,(1)}$. The set $B_{\flip}(U;62)$ is a finite rigid set for the pair $(S,S')$. 
\end{corollary}

\begin{proof}[Proof of Proposition \ref{prop: remaining case specific example}.]
The flip graph of $S'\cong S_{1,0,(1)}$ is pictured in Figure \ref{figure: projection map for remaining case}. It is isomorphic to a subgraph of the lexicographical product $T_3 \cdot \mathbb{R}$, where $T_3$ is the trivalent tree. We can understand the structure of $\flipp$ more fully by viewing it in relation to $\flip$. Say that $p$ is the marked point on $S$. There is a map $f:S' \rightarrow S$ such that $f|_{\text{int}(S')}:\text{int}(S') \rightarrow S\backslash \{p\}$ is a homeomorphism and $f(\partial S') = p$. This map induces a projection $\pi: \flipp \rightarrow \flip$. The fiber $\pi^{-1}(W)$ of a vertex $W$ in $\flip$ is an isometrically embedded copy of $\mathbb{R}$. See Figure \ref{figure: projection map for remaining case}.

\begin{figure}[h]
    \labellist
        \pinlabel {$\flipp$} at 94 10
        \pinlabel {$\flip$} at 306 10
        \pinlabel {$l_0$} at 75 133
        \pinlabel {$l_1$} at 110 133
        \pinlabel {$l_2$} at 130 165
        \pinlabel {$l_3$} at 130 20
        \pinlabel {$l_{-1}$} at 59 20
        \pinlabel {$l_{-2}$} at 58 155
        \pinlabel {$\pi(l_0)$} at 269 89.5
        \pinlabel {$\pi(l_1)$} at 340 89.5
        \pinlabel {$\pi(l_2)$} at 357 140
        \pinlabel {$\pi(l_3)$} at 354 42
        \pinlabel {$\pi(l_{-2})$} at 257 138
        \pinlabel {$\pi(l_{-1})$} at 258 43
    \endlabellist
    \centering
    \includegraphics[scale=.98]{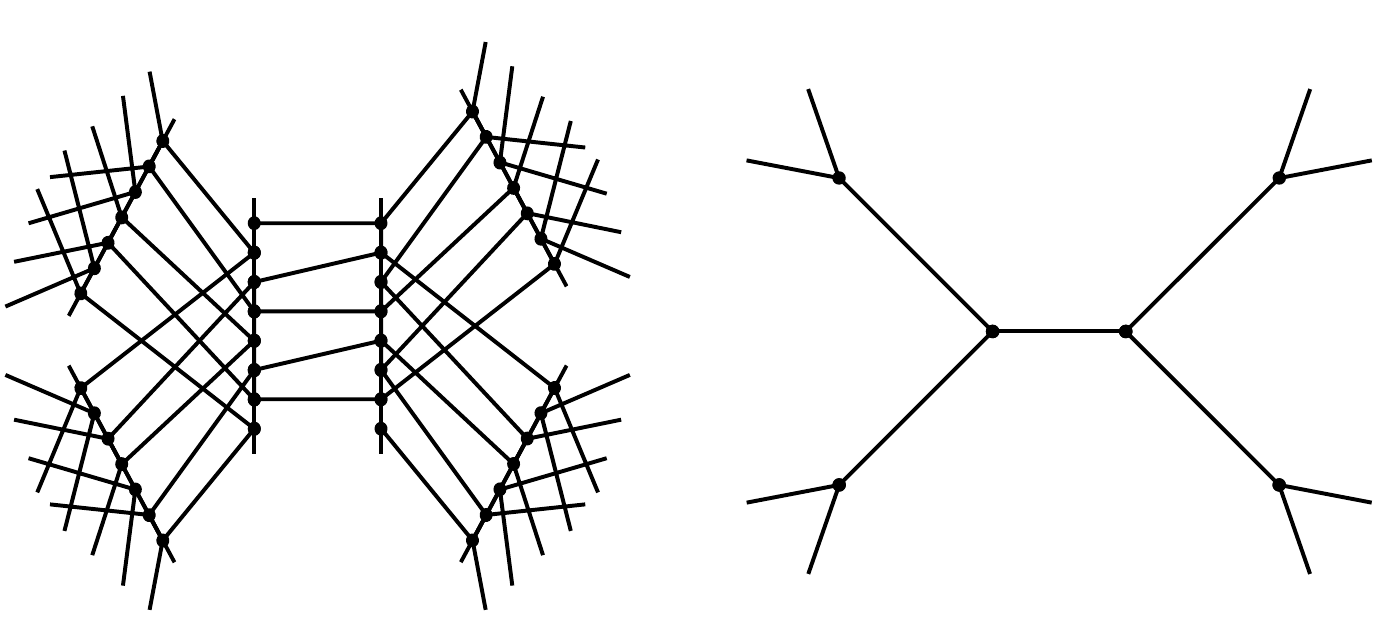}
    \caption{Flip graphs of $S' \cong S_{1,0,(1)}$ (left) and $S \cong S_{1,1}$ (right). Each $l_i$ in $\flipp$ corresponds to an isometrically embedded copy of  $\mathbb{R}$ which maps to a single vertex in $\flip$ under the projection map $\pi$. }
    \label{figure: projection map for remaining case}
\end{figure}
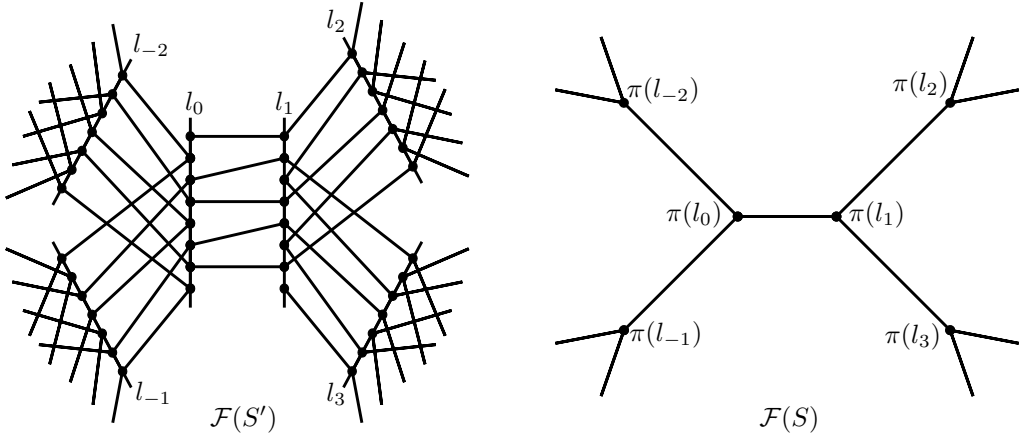

Now fix a vertex $U$ in $\flip$ and assume that there exists an injective homomorphism $\phi:B_{\flip}(U; n) \rightarrow \flipp$ for some $n\in \N$. We will consider the image of $Sph_{\flip}(U;n) = \{V\in \flip:d_{\flip}(U,V)=n\}$, the sphere of radius $n$ in $\flip$ centered at $U$, under the map $\pi \circ \phi$. We use counting arguments to show that $n\leq 61$. This implies that if $n=62$, there does not exist an injective homomorphism $\phi:B_{\flip}(U; n) \rightarrow \flipp$.

For $n \in \N$, define 
\[A(n, \phi)= \frac{1}{|Sph_{\flip}(U;n)|} \displaystyle \sum_{V\in Sph_{\flip}(U;n)} d_{\flip}((\pi\circ \phi)(U),(\pi\circ \phi)(V)),\]
which is the average distance between the images of an element in $Sph_{\flip}(U;n)$ and $U$ under $\pi\circ \phi$. Recall that $\flip$ is locally finite, so $Sph_{\flip}(U;n)$ contains finitely many elements. We find an upper bound for $A(n,\phi)$ for a fixed $n$, independent of the injective homomorphism $\phi$. 

Call an edge in $\flipp$ \textit{type 0} if its two endpoints project to the same point under $\pi$. Call it \textit{type 1} otherwise. If $V\in Sph_{\flip}(U;n)$, there is a unique path $\gamma$ in $\flip$ of length $n$ from $U$ to $V$. The number of type 1 edges in $\phi(\gamma)$ is equal to the length of $(\pi\circ \phi)(\gamma)$, which gives an upper bound for $d_{\flip}((\pi\circ \phi)(U),(\pi\circ \phi)(V))$. Each vertex in $\flipp$ is the endpoint of exactly two type 0 edges and two type 1 edges. Thus, for each $1\leq k \leq n$, if we assume the maximum number of type 1 edges are present in $\phi(B_{\flip}(U;k-1))$, then there can be at most $2^k$ type 1 edges in $\phi(B_{\flip}(U;k)\backslash B_{\flip}(U;k-1))$. Further, an edge in $B_{\flip}(U;k)\backslash B_{\flip}(U;k-1)$ is on the path of length $n$ from $U$ to $V$ for exactly $2^{n-k}$ elements $V\in Sph_{\flip}(U;n)$. Then we can bound $A(n,\phi)$ from above by re-writing the sum as a sum over type 1 edges at distance $k$ from $U$ for each $1\leq k \leq n$. Doing so yields
\[A(n, \phi) \leq \frac{1}{3\cdot 2^{n-1}} \displaystyle \sum_{k=1}^{n} 2^k \cdot 2^{n-k} = \frac{2n}{3} . \]

Then under $\pi\circ \phi$, at least $\frac{1}{5}$ of the elements in $Sph_{\flip}(U;n)$ map into $B_{\flip}\left((\pi\circ \phi)(U);\frac{5}{6} n)\right)$. Otherwise, we would have 
\[A(n,\phi)>\frac{\frac{5}{6} n \cdot \frac{4}{5}|Sph_{\flip}(U;n)|}{|Sph_{\flip}(U;n)|} = \frac{2}{3}\cdot n \geq A(n,\phi),\]
a contradiction. Further, since for any vertex $W$ in $\flip$, the pre-image $\pi^{-1}(W)$ is an isometrically embedded copy of $\mathbb{R}$ in $\flipp$, we see that $\pi$ sends at most $2n+1$ elements from $\phi(B_{\flip}(U;n))$ to the same point. Then 
\begin{align*}
\frac{1}{5} \cdot 3 \cdot 2^{n-1} 
    &=      \frac{1}{5}|Sph_{\flip}(U;n)|\\
    &\leq   (2n+1) \left|B_{\flip}\left((p\circ \phi)(U),\frac{5}{6}n \right)\right| \\ 
    &=  (2n+1)(3\cdot 2^{5n/6}-2) .
\end{align*}
However, this inequality holds only when $n\leq 61$, which gives us the desired result.
\end{proof}

\appendix 
\section{Images of Flip Graphs of Selected Surfaces} \label{appendix: images of flip graphs}

The images below were created by manually drawing triangulations of each surface $S$ and determining which triangulations are related by flips. Since the degree of a vertex in the flip graph is equal to the number of flippable arcs in the triangulation which that vertex represents, it was straightforward to check when all edges incident to a vertex had been drawn. In the figures below, every vertex pictured has (at least a portion of) each incident edge drawn. 

Further, we can guarantee that the portion of each flip graph $\flip$ drawn below is enough so that for any vertex $T$ in $\flip$, there is an automorphism $\phi: \flip \rightarrow \flip$ taking the pictured portion of $\flip$ to a neighborhood of $T$. This is due to the fact that there are finitely many homeomorphism classes of triangulations of a surface $S$ and that these homeomorphisms induce automorphisms $\flip \rightarrow \flip$. For each surface whose flip graph is pictured below, the number of homeomorphism classes of triangulations of the surface is quite small. Indeed, for $S_{1,0,(1)}$ there is only one class. Of the surfaces whose flip graphs are pictured below, the greatest number of homeomorphism classes is six for $S_{0,4}$. Thus, we were able to guarantee that the portion of each flip graph pictured below contains a neighborhood of a vertex which represents a triangulation from each homeomorphism class of triangulations of the surface. We say a bit more below about the larger structure of each flip graph pictured.



In the following four images (the flip graphs of $S_{0,0,(1,2)}$, $S_{0,2,(1)}$, $S_{0,2,(2)}$, and $S_{0,0,(2,2)}$), the flip graph is quasi-isometric to a line and translations of the portion pictured below will generate all of $\flip$ in each case. 

\begin{figure}[H]
    \labellist
        \pinlabel {$\cdots$} at -8 46
        \pinlabel {$\cdots$} at 192 46
        \pinlabel {$U$} at 99.5 46
    \endlabellist
    \centering
    \includegraphics{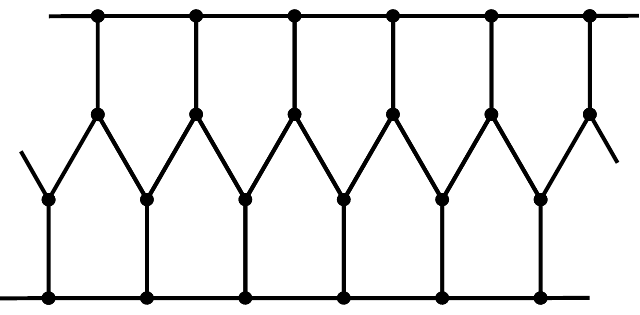}
    \caption{Flip graph of $S_{0,0,(1,2)}$}
\end{figure}

\begin{figure}[H]
    \labellist
        \pinlabel {$\cdots$} at -4 55.5
        \pinlabel {$\cdots$} at 175 55.5
        \pinlabel {$U$} at 77 62
    \endlabellist
    \centering
    \includegraphics{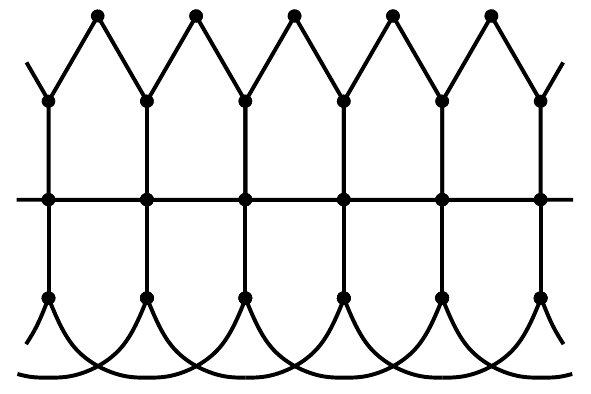}
    \caption{Flip graph of $S_{0,2,(1)}$}
\end{figure}

\begin{figure}[H]
    \labellist
        \pinlabel {$\cdots$} at -5 114
        \pinlabel {$\cdots$} at 374 114
        \pinlabel {$U$} at 154 114
    \endlabellist
    \centering
    \includegraphics{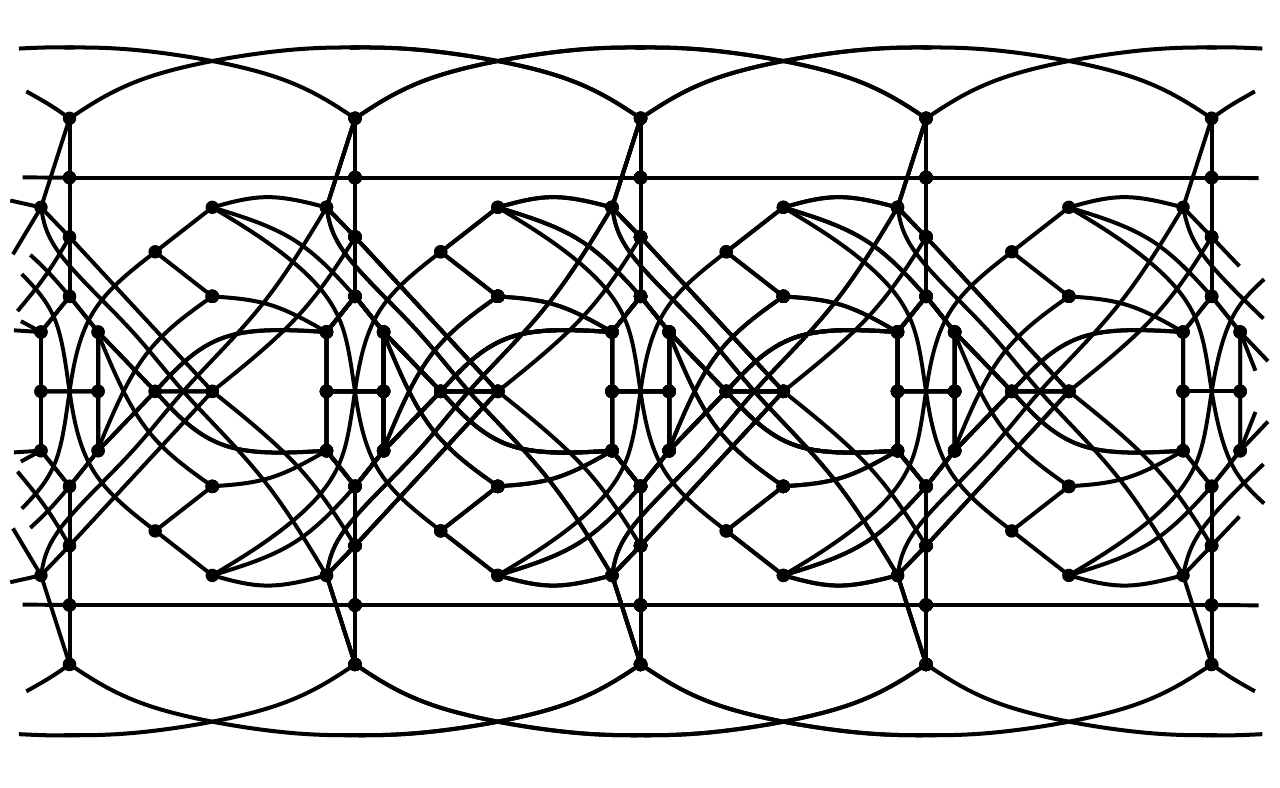}
    \caption{Flip graph of $S_{0,2,(2)}$}
\end{figure}

\begin{figure}[H]
    \labellist
        \pinlabel {$\cdots$} at 0 84
        \pinlabel {$\cdots$} at 270 84
        \pinlabel {$U$} at 137 85
    \endlabellist
    \centering
    \includegraphics{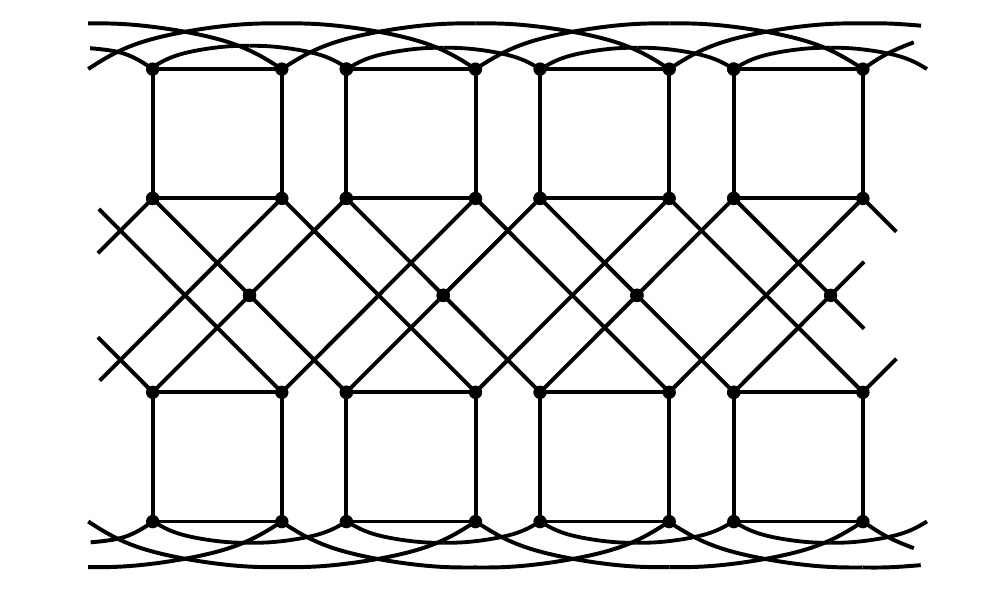}
    \caption{Flip graph of $S_{0,0,(2,2)}$}
\end{figure}

For $S=S_{1,0,(1)}$, the flip graph $\flip$ is isomorphic to a subgraph of the lexicographical product $T_3 \cdot \mathbb{R}$, where $T_3$ is the trivalent tree. There is a simplicial projection $\mathcal{F}(S_{1,0,(1)}) \rightarrow T_3$ which is induced by ``capping'' the boundary component of $S_{1,0,(1)}$. Each fiber of this projection is isomorphic to $\mathbb{R}$. See Section \ref{section: remaining case} for more details.

\begin{figure}[H]
    \labellist
        \pinlabel {$U$} at 143 137
    \endlabellist
    \centering
    \includegraphics{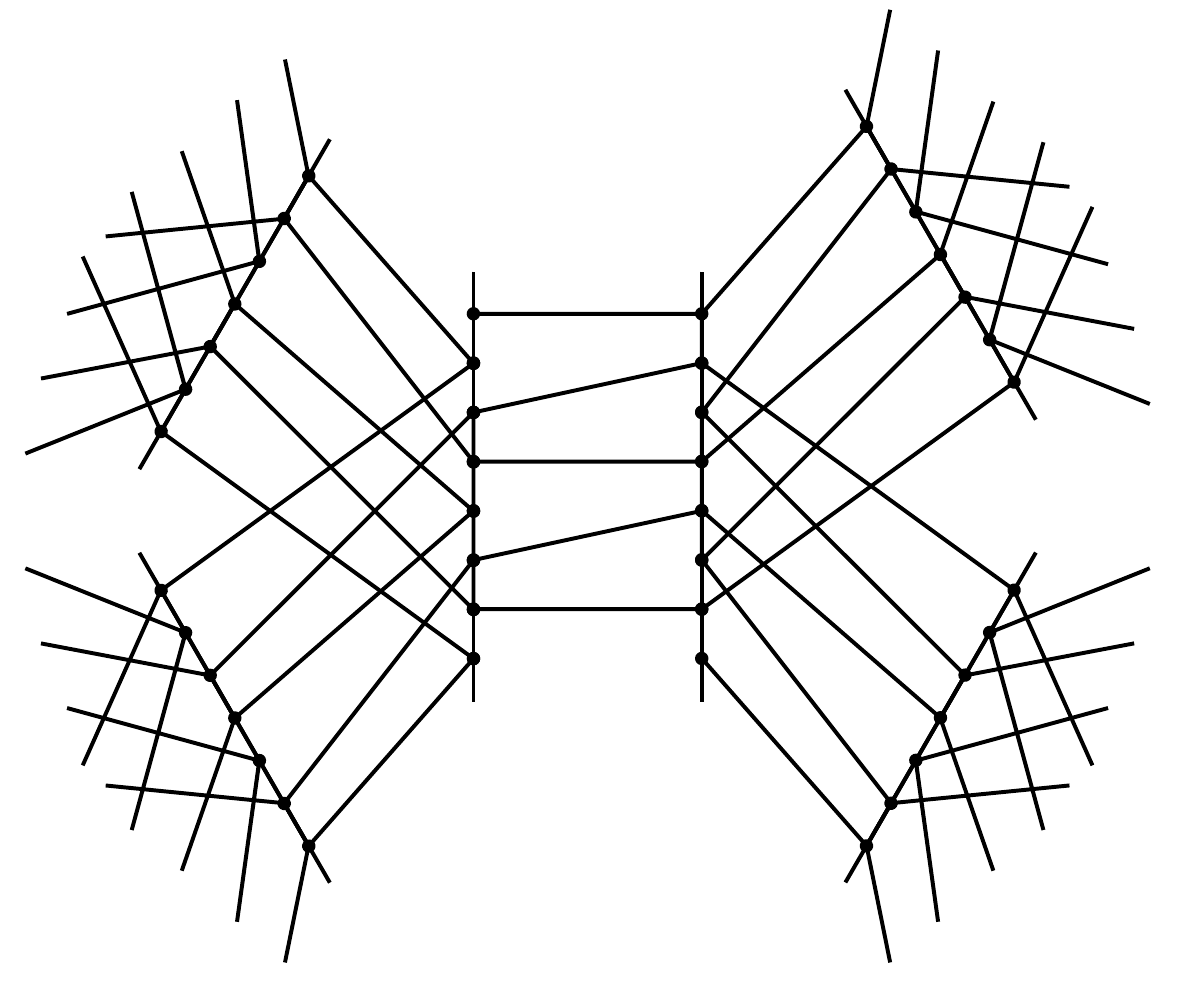}
    \caption{Flip graph of $S_{1,0,(1)}$}
    \label{figure: flip graph simple case 5}
\end{figure}

The flip graph of $S=S_{0,4}$ below should be viewed as a ``thickened trivalent tree.'' The entire flip graph can be generated from this portion by products of a rotation and a reflection of $\flip$. More specifically, if we view $S$ as a regular tetrahedron with a marked point at each vertex, a rotation of order three fixing one of the faces is a homeomorphism of $S$. This homeomorphism induces a rotation of order three of $\flip$ about the point marked $U$. Alternatively, we can view $S$ as a sphere with each marked point on the equator. Then an inversion in the equator is also a homeomorphism of $S$. This homeomorphism induces an automorphism which reflects $\flip$ across the line $l$. The union of the images of the portion of $\flip$ pictured below under products of these two automorphisms yields all of $\flip$.

\begin{figure}[H]
    \labellist
        \tiny
        \pinlabel {$U$} at 397.5 436
        \normalsize
        \pinlabel {$l$} at 425 630
    \endlabellist
    \centering
    \includegraphics[scale=.49]{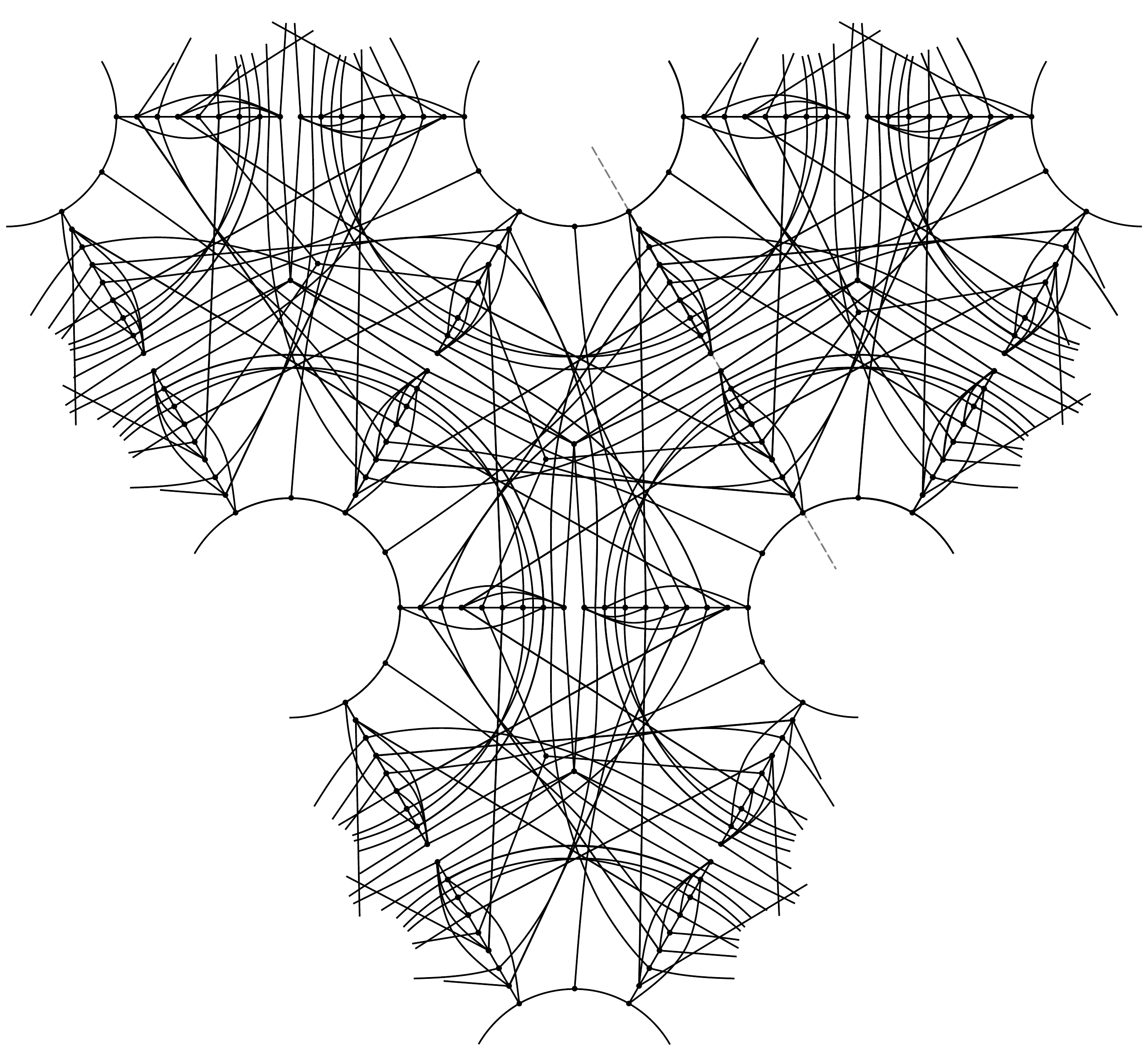}
    \caption{Flip graph of $S_{0,4}$}
\end{figure}

\printbibliography

\end{document}